\documentclass{amsart}

\usepackage{amssymb}
\usepackage{amsthm}
\usepackage{amsmath}
\usepackage[all]{xy}
\usepackage{graphicx}
\usepackage{pifont}
\usepackage{txfonts}
\usepackage{multirow}

\theoremstyle{plain}
\newtheorem{theorem}{Theorem}[section]
\newtheorem*{theorem*}{Theorem}

\newtheorem*{maintheorem*}{Main Theorem}
\newtheorem{proposition}[theorem]{Proposition}
\newtheorem{corollary}[theorem]{Corollary}
\newtheorem{lemma}[theorem]{Lemma}

\newtheorem*{conjecture*}{Conjecture}

\theoremstyle{definition}
\newtheorem{definition}[theorem]{Definition}
\newtheorem*{definition*}{Definition}
\newtheorem*{example*}{Example}
\newtheorem*{notation*}{Notation}
\newtheorem*{notation-conv*}{Notation and convention}
\newtheorem*{convention*}{Convention}
\theoremstyle{remark}
\newtheorem{remark}[theorem]{Remark}

\def\co{\colon\thinspace}

\newcommand{\Z}{{\mathbb Z}}

\newcommand{\C}{{\mathbb C}}
\newcommand{\F}{{\mathbb F}}

\newcommand{\SL}[1][2]{{\mathrm{SL}_{#1}(\C)}}
\newcommand{\GL}{\mathrm{GL}}
\newcommand{\sll}{\mathfrak{sl}_2(\C)}
\newcommand{\sllt}{\C(t) \otimes \sll}
\newcommand{\matrixE}{\begin{pmatrix} 0 & 1 \\ 0 & 0\end{pmatrix}}
\newcommand{\matrixH}{\begin{pmatrix} 1 & 0 \\ 0 & -1\end{pmatrix}}
\newcommand{\matrixF}{\begin{pmatrix} 0 & 0 \\ 1 & 0\end{pmatrix}}

\newcommand{\sqrtrep}{\hbox{\small $(\! \sqrt{-1} \,)^{[\,\cdot\,]}$}}
\newcommand{\signrep}{(-1)^{[\, \cdot \,]}}

\newcommand{\trace}{{\rm tr}\,}

\newcommand{\I}{\mathbf{1}}
\newcommand{\bm}[1]{\mbox{\boldmath{$#1$}}}

\newcommand{\bnd}[1]{\partial_{#1}}
\newcommand{\basisM}[1][i]{\mathbf{c}^{#1}}
\newcommand{\basisH}[1][i]{\mathbf{h}^{#1}}
\newcommand{\basisBt}[1][i]{\tilde{\mathbf{b}}^{#1}}
\newcommand{\basisHt}[1][i]{\tilde{\mathbf{h}}^{#1}}

\newcommand{\knotexterior}{E_K}
\newcommand{\coverExterior}{\widehat{\knotexterior}}
\newcommand{\boundaryTorus}{\partial \knotexterior}
\newcommand{\bpt}{pt}

\newcommand{\univcover}[1]{{\widetilde #1}}
\newcommand{\lift}[1]{\tilde{#1}}

\newcommand{\knotgroup}{\pi_1(\knotexterior)}
\newcommand{\liftknotgroup}{\pi_1(\coverExterior)}
\newcommand{\piDB}{\pi_1(\Sigma_2)}
\newcommand{\nclos}[1]{\langle\!\langle \, #1 \, \rangle\!\rangle}

\newcommand{\repsp}[1]{R(#1, \SL)}

\newcommand{\Tor}[2]{\mathop{\mathrm{Tor}}\nolimits (#1;#2)}
\newcommand{\TorMV}[1]{\mathop{\mathrm{Tor}}\nolimits (#1)}
\newcommand{\TorCpx}[1]{\mathop{\mathrm{Tor}}\nolimits (#1)}

\newcommand{\twistedAlex}[3][t]{\Delta_{#2, \alpha \otimes #3}(#1)}
\newcommand{\twistedAlexAd}[3][t]{\Delta_{#2, {\alpha \otimes Ad \circ #3}}(#1)}

\newcommand{\ie}{i.e.,\,}
\begin{document}


\title[Twisted Alexander, Character varieties, R--torsion of double branched covers]{
  Twisted Alexander polynomials, character varieties and Reidemeister torsion of
  double branched covers
}

\author{Yoshikazu Yamaguchi}

\address{Department of Mathematics,
  Tokyo Institute of Technology
  2-12-1 Ookayama, Meguro-ku Tokyo, 152-8551, Japan}
\email{shouji@math.titech.ac.jp}

\date{\today}


\begin{abstract}
  We give an extension of Fox's formula of the Alexander polynomial
  for double branched covers over the three--sphere.
  Our formula provides the Reidemeister torsion of a double branched cover
  along a knot for a non--trivial one dimensional representation
  by the product of two factors derived from the knot group.
  One of the factors is determined by the twisted Alexander polynomial and the other
  is determined by a rational function on the character variety.
  As an application, we show that these products distinguish
  isotopy classes of two--bridge knots up to mirror images.
\end{abstract}


\maketitle

\section{Introduction}
This paper is intended as an extension of Fox's formula of
the Alexander polynomial for knots in the theory of Reidemeister torsion.
Fox's formula is an application of the Alexander polynomial,
which gives a bridge between the three--dimensional topology and the knot theory.
This means that
Fox's formula provides a computation method to determine the order of
the first homology group of a finite cyclic cover $\Sigma_n$ over
$S^3$ branched along a knot $K$ when $\Sigma_n$ is a rational homology sphere.
It is expressed as
$$
|H_1(\Sigma_n;\Z)|
=\prod_{\ell=0}^{n-1} \Delta_K(e^{\frac{2\pi\sqrt{-1} \ell}{n}})
$$
where $\Delta_K(t)$ is the Alexander polynomial of $K$.
This formula has been extended to links and finite abelian branched covers,
which is due to~\cite{MayberryMurasugi82}.
Fox's formula was also extended for links and finite abelian branched covers over
integral homology spheres 
as an equality of Reidemeister torsions
by J.~Porti~\cite{Porti:MayberryMurasugi}. 

From the viewpoint of Reidemeister torsion,
we can regard Fox's formula as a framework connecting 
two Reidemeister torsions of a cyclic branched cover
and a knot exterior.
In such framework, we start with adopting the definition of
the Alexander polynomial as the order of the first homology group of 
a knot exterior,
whose coefficient of the Laurent polynomial ring $\Z[t,t^{-1}]$.
We deduce Fox's formula from the Alexander polynomial of 
lift $\widehat{K}$ of a knot $K$ in a cyclic branched cover $\Sigma_n$.
We can also define the Alexander polynomial of $\widehat{K}$ as the order of 
the first homology group of the knot exterior with coefficient $\Z[t,t^{-1}]$.
The knot exterior of $\widehat{K}$ is the cyclic cover over the knot exterior of the given knot $K$.
It is known that the following relation holds
between those orders of homology groups with the coefficient $\Z[t, t^{-1}]$
(see~\cite[Theorem~$1.9.2$]{Turaev86:RtorsionKnotTheory}):
$$
\Delta_{\widehat{K}} (t) 
=\prod_{\ell=0}^{n-1} \Delta_K(e^{\frac{2\pi\sqrt{-1} \ell}{n}} t)
$$
where $\Delta_{\widehat{K}}(t)$  is the Alexander polynomial of the knot in $\Sigma_n$.
Evaluating at $t=1$,
the orders $\Delta_{\widehat{K}}(t)$ and $\Delta_K(t)$ of the first homology groups for the knot exteriors
turn into the orders of $H_1(\Sigma_n;\Z)$ and $H_1(S^3;\Z)$.
In the above situation, the Alexander polynomials and 
the order of $H_1(\Sigma_n;\Z)$
can be regarded as Reidemeister torsions
(we refer to~\cite{Turaev86:RtorsionKnotTheory, Porti:MayberryMurasugi}).
To be more precise, the order of $H_1(\Sigma_n;\Z)$ is defined as the Reidemeister torsion by
the trivial $\GL_1(\C)$-representation of $\pi_1(\Sigma_n)$.

The purpose of this paper is to provide an extension of the framework for Fox's formula 
by changing the trivial $\GL_1(\C)$-representation in the Reidemeister torsion of $\Sigma_2$
to non--trivial $\GL_1(\C)$-representations.
We will discuss our extension for double branched covers $\Sigma_2$.
Since Fox's formula shows that every double branched cover is a rational homology sphere,
our main concern is double branched covers $\Sigma_2$ with  non--trivial $H_1(\Sigma_2;\Z)$.
It is natural to begin our extension with the twisted Alexander polynomial 
as a twisted topological invariant corresponding to the Alexander polynomial.
Here a twisted topological invariant means a topological invariant defined by 
linear representations of fundamental groups.

Our extension consists of three steps:
First we need to find homomorphisms of the knot group,
corresponding to a non--trivial $\GL_1(\C)$-representation of $\pi_1(\Sigma_2)$. 
From the author's previous work~\cite{NagasatoYamaguchi},
we can choose irreducible {\it metabelian} representation of the knot group into $\SL$
as corresponding homomorphisms.
In the second step, we have to express the twisted Alexander polynomial of 
the knot in $\Sigma_2$.
We can use the similar formula to the Alexander polynomial,
according to~\cite{DuboisYamaguchi:abelianCovering}.
Last, we must consider an appropriate evaluation of the twisted Alexander polynomial of the knot in $\Sigma_2$.
Actually, in the evaluation, we need an additional term to correct the evaluation of the twisted Alexander polynomial.
Roughly speaking, the special value of twisted Alexander polynomial gives a Reidemeister torsion
defined by some cotangent vector of the character variety
but this cotangent vector does not coincides with the desired framing to give 
the Reidemeister torsion for $\Sigma_2$ and $\xi$. 
To resolve such difference between cotangent vectors on the character variety, 
we need to a rational function, which measures the ratio of two cotangent vectors, on the character variety.
As a result, we will obtain the following equality
(for precise statement, we refer to~Theorem~\ref{thm:main_theorem})
to compute the Reidemeister torsion for
$\Sigma_2$ and a non--trivial $\GL_1(\C)$-representation $\xi$:
$$
|\Tor{\Sigma_2}{\C_{\xi}}|^2
= |\,
P(1)^2 \cdot
F([\rho'])\,
|,
\quad
P(1)=
\lim_{t \to 1}
\frac{
  \twistedAlex[\sqrt{-1} t]{\knotexterior}{\rho}
}{
  t^2-1
}
$$
where
$\twistedAlex{\knotexterior}{\rho}$ denotes the twisted Alexander polynomial of $K$
for an irreducible $\SL$-representation of $\knotgroup$ corresponding to $\xi$,
which is called {\it metabelian},
and $F$ denotes a rational function on the $\SL$-character variety 
and $F([\rho'])$ is the value of $F$ at the conjugacy class of
another irreducible metabelian representation of 
$\knotgroup$ associated with $\rho$ by the adjoint action on the Lie algebra $\sll$.
This theorem allows us to compute the Reidemeister torsion for $\Sigma_2$ and $\xi$ by 
the knot group and $\SL$-representations.

In the last section, we will see an application of our main theorem, which shows that 
we can distinguish two--bridge knots by the twisted Alexander polynomial and 
the rational functions $F$ on the character varieties,
up to mirror images.
This is due to the two facts:
\begin{enumerate}
\item
  We have the one--to--one correspondence between 
  isotopy classes of two--bridge knots and homeomorphism classes of 
  double branched covers, \ie lens spaces.
\item We can derive the condition to classify homeomorphism classes of lens spaces
  from the values of Reidemeister torsion of lens spaces as in our main theorem.
\end{enumerate}
We will also discuss how to compute the rational functions $F$ on character varieties
for two--bridge knots.

Our method also works for knots in integral homology three spheres.
For the simplicity, we will focus on knots in $S^3$.

\medskip

This paper is organized as follows.
In Section~\ref{sec:preliminaries},
we review the Reidemeister torsion associated to twisted chain complexes with 
nontrivial homology groups and properties of character varieties,
which are needed throughout our observation.
We discuss twisted chain complexes and the associated Reidemeister torsion in details  
and show our main theorem  
which gives a connection between the twisted Alexander polynomial of a knot
and the Reidemeister torsion of double branched cover along the knot 
under some conditions
in Section~\ref{sec:main_results}.
Last, Section~\ref{sec:application} reveals that all two--bridge knots 
satisfy the conditions required in our main theorem and 
we can obtain the numerical invariant which classifies 
two--bridge knots up to mirror images.

\section{Preliminaries}
\label{sec:preliminaries}
\subsection{Torsion for chain complexes}
{\it Torsion} is an invariant defined for a based chain complexes.
We denote by $C_*$ the {\it based} chain complex:
$$
C_*: 0 \to C_n \xrightarrow{\bnd{n}}
C_{n-1} \xrightarrow{\bnd{n-1}}
\cdots \xrightarrow{\bnd{2}}
C_1 \xrightarrow{\bnd{1}}
C_0 \to 0
$$
where each chain module $C_i$ is a vector space over a field $\F$
and equipped with a basis $\basisM$. 
The chain complex $C_*$ also has a basis determined by the 
boundary operators $\bnd{i}$,
which arises from the following decomposition of chain modules.
Roughly speaking, torsion provides a property of a chain complex
in the difference between a given basis and new one determined by the boundary operators.

For each boundary operator $\bnd{i}$,
let $Z_i \subset C_i$ denote the kernel and
$B_i \subset C_{i-1}$ the image by $\bnd{i}$.
The chain module $C_i$ is expressed as the direct sum of $Z_i$ and the lift of $B_i$,
denoted by $\lift{B}_i$.
Moreover we can decompose the kernel $Z_i$ into
the direct sum of $B_{i+1}$ and the lift $\lift{H}_i$ of homology group $H_i(C_*)$:
\begin{align*}
  C_i
  &= Z_i \oplus \lift{B}_i \\
  &= \bnd{i+1}\lift{B}_{i+1} \oplus \lift{H}_i \oplus \lift{B}_i 
\end{align*}
where $B_{i+1}$ is written as $\bnd{i+1}\lift{B}_{i+1}$.

We denote by $\basisBt$ a basis of $\lift{B}_{i+1}$.
Choosing a basis $\basisH$ of the $i$-th homology group $H_i (C_*)$,
we can take a lift $\basisHt$ of $\basisH$ in $C_i$.
Then the set $\bnd{i+1}(\basisBt[i+1]) \cup \basisHt \cup \basisBt$
forms a new basis of the vector space $C_i$.
We define the {\it torsion} of $C_*$ as
the following alternating product of determinants of base change matrices:
\begin{equation}
  \label{eqn:def_torsion_complex}
  \mathrm{Tor}(C_*, \basisM[*], \basisH[*])
  = \prod_{i \geq 0} 
  \left[
    \bnd{i+1}(\basisBt[i+1]) \cup \basisHt \cup \basisBt / \basisM
    \right]^{(-1)^{i+1}}
  \in \F^* = \F \setminus \{0\}
\end{equation}
where $[ \bnd{i+1}(\basisBt[i+1]) \cup \basisHt \cup \basisBt / \basisM ]$
denotes the determinant of base change matrix from
$\basisM$ to $\bnd{i+1}(\basisBt[i+1]) \cup \basisHt \cup \basisBt$.

Note that the right hand side is independent of the choice of bases $\basisBt$ and the lift of $\basisH$.
The alternating product in~\eqref{eqn:def_torsion_complex} is determinant by the based chain complex
$(C_*, \basisM[*])$ and the basis $\basisH[*] = \cup_{i \geq 0} \basisH$.

\subsection{Reidemeister torsion for CW--complexes}
We will consider torsion of {\it twisted chain complexes} given by 
a CW--complex and a representation of its fundamental group in this paper. 
Let $W$ denote a finite CW--complex and $(V, \rho)$ a representation of $\pi_1(W)$,
which means $V$ is a vector space over $\C$ and $\rho$ is a homomorphism from $\pi_1(W)$ into $\GL(V)$
and is referred as a $\GL(V)$-representation $\rho$.
\begin{definition}
  We define the twisted chain complex $C_*(W; V_\rho)$ which consists of the twisted chain module as:
  $$
  C_i (W; V_\rho) :=V \otimes_{\Z[\pi_1(W)]} C_i (\univcover{W};\Z)
  $$
  where $\univcover{W}$ is the universal cover of $W$ and $C_i(\lift{W};\Z)$
  is a left $\Z[\pi_1(W)]$-module given by the covering transformation of $\pi_1(W)$.
  In taking the tensor product, we regard $V$ as a right $\Z[\pi_1]$-module under the homomorphism $\rho^{-1}$.
  We identify a chain $\bm{v} \otimes \gamma c$ with $\rho(\gamma)^{-1}(\bm{v}) \otimes c$
  in $C_i (W; V_\rho)$.
\end{definition}
We call $C_*(W;V_\rho)$ the twisted chain complex with the coefficient $V_\rho$ and denote by $H_*(W;V_\rho)$
the homology group, which is called {\it the twisted homology group}.
We will drop the subscript $\rho$ for simplicity when there exists no risk of ambiguity.

Choosing a basis of the vector space $V$, we give a basis of the twisted chain complex $C_*(W; V_\rho)$.
To be more precise, let $\{e^{i}_1, \ldots, e^{i}_{m_i}\}$ be the set of $i$-dimensional cells of $W$ and 
$\{\bm{v}_1, \ldots, \bm{v}_{d}\}$ a basis of $V$ where $d = \dim_{\F} V$. 
Choosing a lift $\lift{e}^{i}_j$ of each cell and taking tensor product with the basis of $V$,
we have the following basis of $C_i(W;V_\rho)$:
$$
\basisM(W;V)=
\{
\bm{v}_1 \otimes \lift{e}^{i}_1, \ldots, \bm{v}_d \otimes \lift{e}^{i}_1,
\ldots,
\bm{v}_1 \otimes \lift{e}^{i}_{m_i}, \ldots, \bm{v}_d \otimes \lift{e}^{i}_{m_i}
\}.
$$

To define the Reidemeister torsion for $W$ and $(V, \rho)$, we require that $V$ has a inner product and 
the basis $\{\bm{v}_1, \ldots, \bm{v}_{d}\}$ is orthonormal.
Regarding $C_*(W; V_\rho)$ as a based chain complex,
we define the Reidemeister torsion for $W$ and $(V, \rho)$
as the torsion of $C_*(W;V_\rho)$ and a basis $\basisH[*]$ of $H_*(W;V_\rho)$, \ie
$$
\mathrm{Tor}(W; V_\rho, \basisH[*] ) = \mathrm{Tor}(C_*(W;V_\rho), \basisM[*](W;V), \basisH[*])
\in \F^* 
$$
up to a factor in $\{\pm \det(\rho(\gamma)) \,|\, \gamma \in \pi_1(W)\}$
since we have many choices of lifts $\lift{e}^{i}_j$
and orders and orientations of cells $e^{i}_j$.

\begin{remark}
  We remark how to avoid indeterminacy of torsion $\mathrm{Tor}(W; V_\rho, \basisH[*])$.
  \begin{itemize}
  \item 
    If the Euler characteristic $\chi(W)$ of a CW--complex $W$ is zero, 
    we can drop the assumption that $\{\bm{v}_1, \ldots, \bm{v}_{d}\}$ is orthonormal. 
    This follows from that the torsion  
    defined by another basis $\{\bm{u}_1, \ldots, \bm{u}_{d}\}$
    is expressed as the product of the torsion $\mathrm{Tor}(C_*(W;V_\rho), \basisM[*](W;V), \basisH[*])$
    defined by $\{\bm{v}_1, \ldots, \bm{v}_{d}\}$
    with the following factor:
    $$
    [ \{\bm{u}_1, \ldots, \bm{u}_{d}\} / \{\bm{v}_1, \ldots, \bm{v}_{d}\} ]^{- \chi(W)}.
    $$
  \item If we choose an $\mathrm{SL}(V)$-representation $\rho$, then
    $\mathrm{Tor}(W; V_\rho, \basisH[*])$ is determined up to a sign.
  \end{itemize}
\end{remark}
It is also worth noting that 
the Reidemeister torsion has an invariance under the conjugation of representations $\rho$.
We often observe the Reidemeister torsion after choosing a representation in the conjugacy class
of a given representation.
When bases $\basisM[*](W;V)$ and $\basisH[*]$ are clear from the context, 
we abbreviate the notation $\mathrm{Tor}(C_*, \basisM[*], \basisH[*])$ to $\TorCpx{C_*}$
and $\mathrm{Tor}(W; V_\rho, \basisH[*])$ to $\Tor{W}{V_\rho}$.
For more details on Reidemeister torsion, we refer to Turaev's book~\cite{Turaev:2000}
and Milnor's survey~\cite{Milnor:1966}.

\subsection{Non--abelian Reidemeister torsion and Twisted Alexander polynomial}
We will mainly observe the Reidemeister torsion of knot exteriors and 
the twisted chain complexes with the coefficient $\sll$.
Let $\knotexterior$ denote the knot exterior $S^3 \setminus N(K)$ where $K$ is a knot in $S^3$
and $N(K)$ is an open tubular neighbourhood.
From now on, we use the symbol $\rho$ for an $\SL$-representation of a knot group $\knotgroup$.
Taking the composition with the adjoint action of $\SL$ on the Lie algebra $\sll$, 
we have a representation $(\sll, Ad \circ \rho)$ of $\knotgroup$ as follows:
\begin{align}
  Ad \circ \rho \co \knotgroup  &\to \SL \to {\rm Aut}(\sll) \label{eqn:def_adjoint_rep}\\
  \gamma  &\mapsto \, \rho(\gamma) \, \mapsto
  Ad_{\rho(\gamma)} \co \bm{v} \mapsto \rho(\gamma)\, \bm{v}\, \rho(\gamma)^{-1}. \notag
\end{align}
It is called the {\it adjoint representation} of $\rho$.
We regard the vector space $\sll$, consisting of trace--free $2 \times 2$-matrices,
as the right $\Z[\knotgroup]$-module via the action $Ad \circ \rho^{-1}$.
The following basis will be referred as the standard basis of $3$-dimensional vector space $\sll$:
\begin{equation}
  \label{eqn:standard_basis_sll}
  \left\{
  E=\matrixE, H=\matrixH, F=\matrixF
  \right\}.
\end{equation}
Note that one can show that $Ad_A$ has the eigenvalues $z^{\pm 2}$ and $1$
when an $\SL$-element $A$ has the eigenvalues $z^{\pm 1}$.
Hence $Ad \circ \rho$ gives an $\SL[3]$-representation of $\knotgroup$.

We will observe the torsion of the twisted chain complex $C_*(\knotexterior;\sll)$
defined by the composition $Ad \circ \rho$.
Since the twisted homology group $H_*(\knotexterior;\sll)$ is non--trivial, 
we need to choose a basis of $H_*(\knotexterior;\sll)$
to define the torsion of $C_*(\knotexterior;\sll)$.
The twisted homology group $H_*(\knotexterior;\sll)$ depends on the choice of 
$\SL$-representations $\rho$.
We follow notion~\cite{JDFourier, YY1} of $\SL$-representations 
concerning a basis of $H_*(\knotexterior;\sll)$, 
which was introduced by J.~Porti~\cite{Porti:1997}.

First, we recall the notion of {\it irreducible} and {\it reducible} representations.
When there exists a proper invariant line in $\C^2$ under the action of 
the image $\rho(\knotgroup)$,
we say that $\rho$ is {\it reducible}.
If an $\SL$-representation $\rho$ is not reducible, then $\rho$ is referred to as being {\it irreducible}.
We are not concerned with reducible $\SL$-representation of $\knotgroup$.
\begin{definition}
  \label{def:regularity_rep}
  Let $\gamma$ be a closed curve on $\boundaryTorus$.
  An $\SL$-representation $\rho$ of $\knotgroup$ is $\gamma$-regular if
  $\rho$ is irreducible and satisfies the following conditions:
  \begin{enumerate}
  \item \label{item:dim_regularity}
    $\dim_\C H_1(\knotexterior;\sll) = 1$;
  \item \label{item:generator_regularity}
    the inclusion map from $\gamma$ into $\knotexterior$
    induces the surjective homomorphism from $H_1(\gamma;\sll)$ onto $H_1(\knotexterior;\sll)$,
    where $H_1(\gamma;\sll)$ is the homology group of twisted chain complex for $\gamma$
    and the restriction of $\rho$ on the subgroup $\langle \gamma \rangle \subset \knotgroup$;
  \item \label{item:parabolic_regularity}
    if $\trace \rho(\gamma') = \pm 2$ for all $\gamma' \in \pi_1(\boundaryTorus)$, then
    $\rho(\gamma) \not = \pm \I$.
  \end{enumerate}
\end{definition}

\begin{lemma}[Lemma~$2.6.5$ in~\cite{YY1}]
  Under the assumption of Definition~\ref{def:regularity_rep}, 
  if $\rho$ is $\gamma$-regular,
  then the induced homomorphism from $H_2(\boundaryTorus;\sll)$ to $H_2(\knotexterior;\sll)$
  is an isomorphism, where $H_2(\boundaryTorus;\sll)$ is the homology group of twisted chain complex
  for $\boundaryTorus$ and the restriction $\rho$ on $\pi_1(\boundaryTorus)$.
\end{lemma}
For an explicit basis of $H_*(\boundaryTorus;\sll)$,
see~\cite[Lemma~$2.6.2$]{YY1} and \cite[Proposition~3.18]{Porti:1997}.
\begin{proposition}
  Let $\gamma$ be a closed loop on $\boundaryTorus$,
  If $\rho$ is $\gamma$-regular, then we can choose 
  the following basis of the twisted homology group $H_*(\knotexterior;\sll)$:
  \begin{equation}
    \label{eqn:basis_gamma_reg}
    H_*(\knotexterior;\sll) =
    \begin{cases}
      [P^\rho \otimes T^2] & (*=2) \\
      [P^\rho \otimes \gamma] & (*=1) \\
      \bm{0} & ({\rm otherwise})
    \end{cases}
  \end{equation}
  where we use the same notation for the lifts of chains to the universal cover
  $\univcover{\knotexterior}$ for the simplicity.
\end{proposition}

Let $\mu$ and $\lambda$ denote a meridian and
a corresponding preferred longitude on $\boundaryTorus$.
We will focus on $\mu$-regular and $\lambda$-regular representation of $\knotgroup$.
We refer to~\cite[Section~$2.6$]{YY1} as an exposition.
\begin{definition}
  We assume that $\rho$ is $\gamma$-regular for a closed loop on $\boundaryTorus$.
  Then we will consider the Reidemeister torsion for $\knotexterior$ and the representation
  $(\sll, Ad \circ \rho^{-1})$
  with the basis as in Eq.~\eqref{eqn:basis_gamma_reg} and write it simply $\Tor{\knotexterior}{\sll}$.
\end{definition}

We also review the twisted Alexander polynomial in the context of Reidemeister torsion.
Let $\C(t)$ be the rational functional field with a variable $t$.
We denote by $\alpha$ the abelianization homomorphism of $\knotgroup$, \ie
$$
\alpha \co \knotgroup \to H_1(\knotexterior;\Z) = \langle t \rangle,
$$
which sends $\mu$ to $t$.

For an $\SL$-representation $\rho$ of $\knotgroup$, 
the tensor product $\alpha \otimes Ad \circ \rho^{-1}$ gives an action of
$\knotgroup$ on the vector space $\C(t) \otimes_{\C} \sll$
over the rational function field $\C(t)$.
Extending the action defined by $\alpha \otimes Ad \circ \rho^{-1}$
to the group ring $\Z[\knotgroup]$ linearly,
we can construct the following chain complex:
$$
C_*(\knotexterior;\C(t) \otimes \sll) =
(\C(t) \otimes \sll) \otimes_{\Z[\knotgroup]} C_*(\univcover{\knotexterior};\Z).
$$

If the chain complex $C_*(\knotexterior;\C(t) \otimes \sll)$ is acyclic,
then the torsion $\TorCpx{C_*(\knotexterior;\C(t) \otimes \sll)}$ is defined as 
en element in $\C(t) \setminus \{0\}$ up to a factor $\pm t^k$ ($k \in \Z$).
According to the observation by~\cite[Section~$4$]{KL}, this torsion
$\TorCpx{C_*(\knotexterior;\C(t) \otimes \sll)}$ can be regarded as
the twisted Alexander polynomial of $K$ and $Ad \circ \rho$,
defined by X-S.~Lin~\cite{Lin01} and M.~Wada~\cite{Wada94}.
\begin{definition}
  Suppose that the chain complex $C_*(\knotexterior;\C(t) \otimes \sll)$ is acyclic.
  Then we call $\TorCpx{C_*(\knotexterior;\C(t) \otimes \sll)}$
  the twisted Alexander polynomial for $K$ and $Ad \circ \rho$ and denote it by
  $\twistedAlexAd{\knotexterior}{\rho}$.
\end{definition}  
\begin{remark}
  By the definition due to Wada, we can compute $\twistedAlexAd{\knotexterior}{\rho}$
  by Fox differential calculus with a presentation of $\knotgroup$.
  The twisted Alexander polynomial is usually easy to compute, compared to
  the Reidemeister torsion $\Tor{\knotexterior}{\sll}$.
\end{remark}

\subsection{Character varieties}
We review the $\SL$-character varieties of knot groups briefly.
It is known that the twisted homology group $H_1(\knotexterior;\sll)$ is isomorphic to
the Zariski cotangent space of the character variety of $\knotgroup$ if $\rho$ is irreducible
and satisfies the conditions~\eqref{item:dim_regularity} \&~\eqref{item:parabolic_regularity}
in Definition~\ref{def:regularity_rep}.
When we consider a base change in $H_1(\knotexterior;\sll)$, 
it is helpful to use the geometric properties and regular functions on the character variety.
We start with the definition of character varieties as sets.
\begin{definition}
  We denote by $\repsp{\knotgroup}$ the set of $\SL$-representations of a knot group $\knotgroup$.
  Choosing an $\SL$-representation $\rho$, we call the following map $\chi_\rho$ the {\it character} of $\rho$:
  \begin{align*}
    \chi_\rho \co \knotgroup &\to \C \\
    \gamma &\mapsto \trace \rho(\gamma).
  \end{align*}
  The set $\{\chi_\rho \,|\, \rho \in \repsp{\knotgroup}\}$ of characters is denoted by $X(E_K)$.
\end{definition}
According to M.~Culler and P.~Shalen~\cite{CS:1983}, 
the set $X(\knotexterior)$ consists of several components which have structures of affine variety.
We call $X(\knotexterior)$ the {\it character variety} of $\knotgroup$.
The character variety is referred as the algebraic--quotient of the representation variety $\repsp{\knotgroup}$
under the action of $\SL$ by conjugation.

We mainly deal with irreducible $\SL$-representations and their characters. 
We will write $X^{\rm irr}(\knotexterior)$ for
the subset in $X(\knotexterior)$, consisting of irreducible characters,
which means the characters of irreducible $\SL$-representations of $\knotgroup$.

\begin{remark}
  By~\cite[Proposition~$1.5.2$]{CS:1983}, we can regard $X^{\rm irr}(\knotexterior)$ as
  the set of conjugacy classes of irreducible $\SL$-representations of $\knotgroup$.
  This is due to~\cite[Proposition~$1.5.2$]{CS:1983} which means that
  if irreducible representations $\rho$ and $\rho'$ give the same character,
  then they are conjugate to each other.
  We can think of $X^{\rm irr}(\knotexterior)$
  as a parameter space of Reidemeister torsion 
  for irreducible $\SL$-representations of $\knotgroup$.
\end{remark}

In general,
we have the following inequalities of dimensions related to the character variety $X(\knotexterior)$:
$$
\dim_\C X(\knotexterior)
\leq \dim_{\C} {T^{Zar}_{\chi_\rho} X(\knotexterior)}
\leq \dim_{\C} H^1(\knotexterior;\sll) (= \dim_{\C} H_1(\knotexterior;\sll))
$$
where $T^{Zar}_{\chi_\rho} X(\knotexterior)$ is the Zariski tangent space at $\chi_\rho$
(we refer to~\cite[Section~$3.1.3$]{Porti:1997}) and 
the last equality is due to the universal coefficient theorem.
The first equality holds when $\chi_\rho$ is a smooth point 
and second equality holds when $\rho$ is irreducible.
\begin{remark}
  According to~\cite[Proposition~$3.5$]{Porti:1997},
  if $\rho$ is irreducible, then
  we can identify $H_1(\knotexterior;\sll)$ with
  the dual space of Zariski tangent space $T^{Zar}_{\chi_\rho} X(\knotexterior)$.
\end{remark}

The structure of $X(\knotexterior)$ as affine variety arises from functions
defined as the evaluation for each $\gamma \in \knotgroup$:
\begin{align}
  \label{eqn:def_trace_function}
  X(\knotexterior) &\to  \C \\
  \chi_\rho &\mapsto \chi_\rho(\gamma)=\trace \rho (\gamma). \notag
\end{align}
The appropriate set of $N$ elements in $\knotgroup$ forms an embedding to the affine space $\C^N$
and its image turns into a closed algebraic set,
which gives a structure of affine variety to $X(\knotexterior)$.
\begin{definition}
  We will write $I_\gamma$ for the function in Eq.~\eqref{eqn:def_trace_function} and
  call it the {\it trace function} of $\gamma \in \knotgroup$.
\end{definition}

We also review the symmetry of $\SL$-character varieties under an involution.
We will focus irreducible $\SL$-representations of knot groups,
whose characters are characterized
as the fixed point set by the involution.
Such involution is given by multiplying the following $\GL_1(\C)$-representation
of $\knotgroup$ into $\{\pm 1\}$:
\begin{align*}
  \signrep \co \knotgroup & \to \{\pm 1\} \subset \C \\
  \gamma &\mapsto (-1)^{[\gamma]}
\end{align*}
where $[\gamma]$ is an integer corresponding to 
the homology class of $\gamma$ under the identification
$H_1(\knotexterior;\Z) \ni \mu \mapsto 1 \in \Z$.
For every $\SL$-representation $\rho$ of $\knotgroup$, 
the product $\signrep \cdot \rho$ also gives an $\SL$-representation.
This multiplication gives an involution on $\repsp{\knotgroup}$ and also induces 
an involution on the character variety $X(\knotexterior)$, which 
sends $\chi_\rho$ to $\chi_{\signrep \rho} = \signrep \chi_\rho$, 
since $\signrep$ commutes with conjugation.
We denote by $\hat{\iota}$ this involution. 
The involution $\hat{\iota}$ has the fixed point set on $X^{\rm irr}(\knotexterior)$ unless
the knot determinant $\Delta_K(-1)$ is $\pm 1$.
\begin{proposition}[Proposition~$3$ in~\cite{NagasatoYamaguchi}]
  \label{prop:fixed_points_involution}
  If the knot determinant $\Delta_K(-1)$ is not equal to $\pm 1$, then
  the involution $\hat{\iota}$ on $X^{\rm irr}(\knotexterior)$ has
  finite fixed points consisting of
  all characters of irreducible metabelian representations,
  whose number is equal to $(|\Delta_K(-1)| -1)/2$.
\end{proposition}
We will review the definition of {\it metabelian} in the subsequent
Section~\ref{subsect:Corresp_Reps}.
Proposition~\ref{prop:fixed_points_involution} means that
every irreducible metabelian representation $\rho$ is conjugate to
the product $\signrep \rho$.

\section{Main results}
\label{sec:main_results}
We will describe the Reidemeister torsion of double branched covers
$\Sigma_2$ over $S^3$ along knots $K$.  It is known that every
double branched cover $\Sigma_2$ is a rational homology three sphere
since the order is determined by $\Delta_K(-1)$.
In this section, we suppose that the homology group $H_1(\Sigma; \Z)$
is non--trivial since we consider the Reidemeister torsion of
$\Sigma_2$ defined by abelian representations of $\pi_1(\Sigma_2)$.
This means that we consider knots whose determinants are not equal to $\pm 1$.
We derive the Reidemeister torsion of $\Sigma_2$ from the
Mayer--Vietoris argument for local systems induced by the following decomposition:
$$
\Sigma_2 = \coverExterior \cup D^2 \times S^1
$$
where $\widehat{E_K}$ is the cyclic $2$--fold cover of $E_K$ and
the meridian disk $\partial D^2 \times \{*\}$ is glued with the lift
of $\mu^2$.
We will see the Reidemeister torsion of $\Sigma_2$ is given by the 
product of Reidemeister torsions of $\widehat{E_K}$ and the Mayer--Vietoris homology exact sequence
in the following Subsection~$\ref{subsect:Mayer_Vietoris_argument}$.
In Subsections~\ref{subsect:Torsion_cyclic_covers_knotexteriors} 
and~\ref{subsect:Torsion_Mayer_Vietoris_sequence},
we will discuss how to derive the Reidemeister torsions of $\widehat{E_K}$ and 
the homology exact sequence with the twisted Alexander polynomial and
the character variety of $\knotgroup$.

\subsection{Correspondence of representations}
\label{subsect:Corresp_Reps}
We consider double branched covers $\Sigma_2$ of $S^3$ with
non--trivial first homology group $H_1(\Sigma_2; \Z)$.  Since the order of
$H_1(\Sigma_2; \Z)$ is finite, we have finite abelian representations
$\xi$ from $\pi_1(\Sigma_2)$ into $\GL_1(\C) = \C \setminus \{0\}$.
Such representation $\xi$ gives a local system of $\Sigma_2$ and we describe 
Mayer--Vietoris exact sequences for local systems determined by $\xi$,
induced from the decomposition
$\Sigma_2 = \coverExterior \cup D^2 \times S^1$
where $\coverExterior$ is the $2$--fold cover over $\knotexterior$ as in the following diagram:
$$
\xymatrix{
  **[l] \widehat{\mu} \subset \coverExterior \ar[r] \ar[d] & \Sigma_2 \ar[d]\\
  **[l] \mu^2 \subset \knotexterior \ar[r]& S^3
}
$$
where $\widehat{\mu}$ is the lift of $\mu^2$.
Since our main concern is a relation between Reidemeister torsions of $\Sigma_2$ and $\knotexterior$,
we will also express a local system of $\coverExterior$ as a direct sum of local systems of $\knotexterior$.

First of all, we recall what kind of representations of the knot group
$\knotgroup$, which corresponds to the $\GL_1(\C)$-representations of
$\piDB$.
\begin{proposition}[Theorem~$1$ in \cite{NagasatoYamaguchi}]
  \label{prop:correspondence_abel_metabel}
  We have the following one to one correspondence between conjugacy classes of
  $\SL$-representations:
  \begin{align*}
    &\{\xi \oplus \xi^{-1} \co \pi_1(\Sigma_2) \to \SL \,|\,
    \xi \co \pi_1(\Sigma_2) \to \GL_1(\C) \} \,/\,\hbox{conj.}\\
    &\overset{\hbox{$1:1$}}{\longleftrightarrow}
    \left\{
    \rho \co \knotgroup \to \SL \, | \,
    \trace \rho (\mu) = 0,\, 
    \rho : \hbox{metabelian}
    \right\} \,/\, \hbox{conj.}
  \end{align*}
  In particular, the number of conjugacy classes 
  is equal to $\frac{1}{2}(|\Delta_K(-1)| - 1) + 1$  where $\Delta_K(t)$ is the Alexander polynomial of $K$.
\end{proposition}
In the context of character varieties, conjugacy classes should be replaced with 
characters.
An $\SL$-representation is called {\it metabelian}
when the image of commutator subgroup $[\knotgroup, \knotgroup]$ is an abelian subgroup in $\SL$.
\begin{remark}
  By~\cite[Lemma~$9$ and Proposition~$4$]{NagasatoYamaguchi} or
  \cite[Proposition~$1.1$]{nagasato07:_finit_of_section_of_sl},
  every irreducible metabelian representation $\rho$
  satisfies that $\trace \rho(\mu) = 0$.
  Every non--trivial homomorphism $\xi$ of $\pi_1(\Sigma_2)$ corresponds to 
  an irreducible metabelian representation of $\knotgroup$.
  The number of conjugacy classes is given by $\frac{1}{2} (|\Delta_K(-1)|-1)$.
\end{remark}

To describe the relation between local systems, we review the map giving the correspondence 
in Proposition~\ref{prop:correspondence_abel_metabel}
(we refer to~\cite[Section~$5$]{NagasatoYamaguchi}).

Let $p$ denote the induced homomorphism from $\liftknotgroup$ to $\knotgroup$.
This homomorphism $p$ sends $\widehat{\mu}$ to $\mu^2$.
The pull--back of $\rho$ gives an $\SL$-representation $p^* \rho$ of $\liftknotgroup$.
Since this representation $p^* \rho$ sends $\widehat{\mu}$ to $-\I$,
we need a sign refinement to give an $\SL$-representation of $\piDB$,
which is induced from the following $1$-dimensional representation:
\begin{align*}
  \left(\! \sqrt{-1} \right)^{[\,\cdot\,]} \co 
  \knotgroup &\to \GL_1(\C) \\
  \gamma &\mapsto \left(\! \sqrt{-1} \right)^{[\gamma]}
\end{align*}
where $[\gamma]$ denotes an integer corresponding to the homology class
under the isomorphism 
$$
H_1(\knotexterior;\Z) \ni
\mu \mapsto 1 \in \Z.
$$

The product of {\small $(\! \sqrt{-1} )^{p^* [\,\cdot\,]}$} and $p^* \rho$ defines 
an $\SL$-representation of $\piDB$. 
Here we identify the fundamental group $\piDB$ with the quotient group
$\liftknotgroup / \nclos{\widehat{\mu}}$
where $\nclos{\widehat{\mu}}$ denotes the normal closure of $\widehat{\mu}$.
\begin{definition} 
  We say that an irreducible metabelian $\SL$-representation $\rho$ corresponds to $\xi$
  if {\small $(\! \sqrt{-1} )^{p^* [\,\cdot\,]}$}$ \cdot p^* \rho$ induces the diagonal representation 
  $\xi \oplus \xi^{-1}$ on $\piDB$.
\end{definition}

We need to consider the relation of local systems of $\Sigma_2$ and
$\knotexterior$ given by the $\SL$-representation $\xi \oplus \xi^{-1}$ and 
the $\GL_2(\C)$-representation {\small $(\! \sqrt{-1} )^{[\,\cdot\,]}$} $\rho$. 
From \cite[Section~$3.2$ and $4.2$]{yamaguchi:twistedAlexMeta}, we can see that 
the adjoint representation defined as Eq.~\eqref{eqn:def_adjoint_rep}
is useful for this purpose as follows.
We can find the $2$-dimensional representation {\small $(\! \sqrt{-1} )^{[\,\cdot\,]}$}
$\rho$ as a direct summand in the adjoint representation for
another irreducible metabelian representation of $\knotgroup$.

\begin{lemma}
  \label{lemma:adjoint_metabelian}
  For every irreducible metabelian $\SL$-representation $\rho$ of the knot group $\knotgroup$,
  there exists another irreducible metabelian representation $\rho'$ and an $\SL$-matrix $C$
  satisfying the following decomposition:
  $$
  Ad \circ \rho' =
  (-1)^{[\, \cdot \,]} \oplus C \, \sqrtrep \rho \, C^{-1}.
  $$
  where the equality holds for the ordered basis $\{H, E, F\}$ of $\sll$.
\end{lemma}

\begin{proof}
  Since $\trace \rho(\mu)=0$, we can choose $C \in \SL$ satisfying 
  $$
  C \rho(\mu) C^{-1} =
  \begin{pmatrix}
    0 & \sqrt{-1} \\
    \sqrt{-1} & 0
  \end{pmatrix}.
  $$
  From~\cite[Proposition~$3.14$]{yamaguchi:twistedAlexMeta}, we can
  decompose the adjoint representation of each irreducible metabelian
  representation into $\psi_1 \oplus \psi_2$ where $\psi_1$ is
  $1$-dimensional representation $(-1)^{[\, \cdot \, ]}$ and $\psi_2$ is $2$-dimensional one.
  To be more precise,
  we can choose a representative in the conjugacy class of
  irreducible metabelian representations, whose adjoint representation decomposes
  $\sll$ into $V_1 \oplus V_2$ where $V_1 = \langle H \rangle$ and $V_2 = \langle E, F \rangle$.

  It also follows from~\cite[Proposition~$3.12$ and the sequel of Eq.~$(11)$]{yamaguchi:twistedAlexMeta} that
  there exists another metabelian $\SL$-representation $\rho'$ such that each summand $\psi_1$ and 
  $\psi_2$ in $Ad \circ \rho'$ send the meridian $\mu$ to the matrices $-1$ and 
  $\begin{pmatrix}
    0 & -1 \\
    -1 & 0
  \end{pmatrix}$ and the commutator subgroup into $\{1\}$ and the abelian subgroup in $\SL$,
  given by the image of $\rho$.
\end{proof}

\begin{remark}
  In Proposition~\ref{lemma:adjoint_metabelian},
  two irreducible metabelian representation $\rho$ and $\rho'$ are usually contained
  in distinct conjugacy classes of $\SL$-representations of $\knotgroup$ since
  they have the different image of the commutator subgroup.
  We refer to~\cite[Section~$5$]{yamaguchi:twistedAlexMeta} for explicit examples.
\end{remark}

\begin{definition}
  \label{def:irred_metarep_form}
  Here and subsequently, we always use the symbol $\rho$ to denote
  an irreducible metabelian $\SL$-representation
  of $\knotgroup$ which sends the meridian $\mu$ to the matrix
  $\begin{pmatrix}
    0 & \sqrt{-1}\\
    \sqrt{-1} & 0
  \end{pmatrix}$.
  We will denote by $\psi_1$ and $\psi_2$
  the $1$-dimensional representation given by $(-1)^{[\, \cdot \,]}$ and
  the $2$-dimensional one given by 
  {\small $(\! \sqrt{-1} )^{[\,\cdot\,]}$} $\rho$ for an irreducible metabelian $\SL$-representation $\rho$
  of $\knotgroup$.
\end{definition}

Under the assumption in Definition~\ref{def:irred_metarep_form},
every metabelian representation $\rho$ sends the commutator subgroup into the maximal abelian subgroup consisting
diagonal matrices in $\SL$.
We summarize the properties of the adjoint representation $Ad \circ \rho'$ in Lemma~\ref{lemma:adjoint_metabelian}.

\begin{remark}
  \label{remark:rest_adjoint_rep}
  Suppose that $\rho$ corresponds to $\xi \co \piDB \to \GL_1(\C)$.
  \begin{itemize}
  \item The pull--back $p^* (Ad \circ \rho')$ sends $\widehat{\mu}$
    to the identity matrix.
  \item The pull--back $p^* (Ad \circ \rho')$ turns out to be an $\SL[3]$-representation 
    $(-1)^{p^* [\, \cdot \,]} \oplus$ {\small $(\! \sqrt{-1} )^{p^*[\,\cdot\,]}$} $p^* \rho$ of $\liftknotgroup$.
    Since $p^* [\, \cdot \, ]$ maps $\liftknotgroup$ onto $2\Z$, the pull--back $p^* (Ad \circ \rho')$ induces
    the diagonal representation $\I \oplus (\xi \oplus \xi^{-1})$ on $\piDB$.
    Here $\I$ denotes the trivial $1$-dimensional representation.
  \item
    Since a preferred longitude $\lambda$ is included in the second commutator subgroup,
    every metabelian representation send $\lambda$ to the identity matrix.
    Hence the restriction of $p^* (Ad \circ \rho')$ on $\pi_1(\partial \coverExterior)$ is
    the trivial $3$-dimensional representation.
  \end{itemize}
\end{remark}
We will identify $\widehat{\mu}$ with $\mu^2$ in $\knotgroup$ under the injective homomorphism $p$.

\subsection{Mayer--Vietoris argument}
\label{subsect:Mayer_Vietoris_argument}
When we consider the local system of $\coverExterior$ induced by the pull--back $p^*(Ad \circ \rho')$,
we can extend this local system to that of $\Sigma_2$ by Lemma~\ref{lemma:adjoint_metabelian}.
We also define local systems on the boundary torus $\partial \coverExterior$ and 
the attached solid torus $\Sigma_2 \setminus \hbox{int}(\coverExterior)$
by the restrictions of local system of $\Sigma_2$.
Note that the coefficients of these local systems are $\sll$.
From the decomposition $\Sigma_2 = \coverExterior \cup D^2 \times S^1$,
we have the short exact sequence of local systems with the coefficient $\sll$:
\begin{equation}
  \label{eqn:short_exact_sll}
  0 \to C_*(\widehat{T^2};\sll)
  \to C_*(\coverExterior;\sll) \oplus C_*(D^2 \times S^1;\sll)
  \to C_*(\Sigma_2;\sll) \to 0
  \tag{$\mathcal{S}$}
\end{equation}
where $\widehat{T^2}$ denotes the boundary torus $\partial \coverExterior$.

We can see that the two chain complexes for $\widehat{T^2}$ and $D^2 \times S^1$
are determined by the usual chain complex with the coefficient $\C$ 
since the restriction of $p^* (Ad \circ \rho')$ on each fundamental group is
trivial as in Remark~\ref{remark:rest_adjoint_rep} .
\begin{lemma}
  In the short exact sequence~\eqref{eqn:short_exact_sll},
  the chain complexes $C_*(\widehat{T^2};\sll)$ and $C_*(D^2 \times S^1;\sll)$
  turns into $\sll \otimes_{\C} C_*(T^2;\C)$ and $\sll \otimes_{\C} C_*(D^2 \times S^1;\C)$.
\end{lemma}

We also show the decomposition of $C_*(\Sigma_2;\sll)$ used in the remain of this subsection.
\begin{lemma}
  \label{lemma:decomp_twistedchain_DB}
  The twisted chain complex $C_*(\Sigma_2;\sll)$ is decomposed into
  the following direct sum;
  $$
  C_*(\Sigma_2;\C) \oplus C_*(\Sigma_2;\C_\xi) \oplus C_*(\Sigma_2;\C_{\xi^{-1}}).
  $$
  The last summand $C_*(\Sigma_2;\C_{\xi^{-1}})$ coincides with
  $C_*(\Sigma_2;\C_{\bar \xi})$ where
  $\bar \xi$ denotes the complex conjugate representation of $\xi$.
\end{lemma}
\begin{proof}
  Since the pull--back $p^* (Ad \circ \rho')$ gives
  the diagonal representation $\I \oplus (\xi \oplus \xi^{-1})$ of $\pi_1(\Sigma_2)$, the
  twisted chain complex $C_*(\Sigma;\sll)$ turns into the direct sum
  $C_*(\Sigma_2;\C) \oplus C_*(\Sigma_2;\C_\xi) \oplus C_*(\Sigma_2;\C_{\xi^{-1}})$.
  Since $H_1(\Sigma_2;\Z)$ is finite, 
  the representation $\xi^{-1}$ agrees with the complex conjugate $\bar \xi$.
\end{proof}

To express the homology groups, let us introduce the notations for cycles in the twisted chain complex
$C_*(T^2;\sll)$ as in 
Figure~\ref{fig:double_cover_torus}:
\begin{figure}[ht]
  \begin{center}
    \includegraphics[scale=.4]{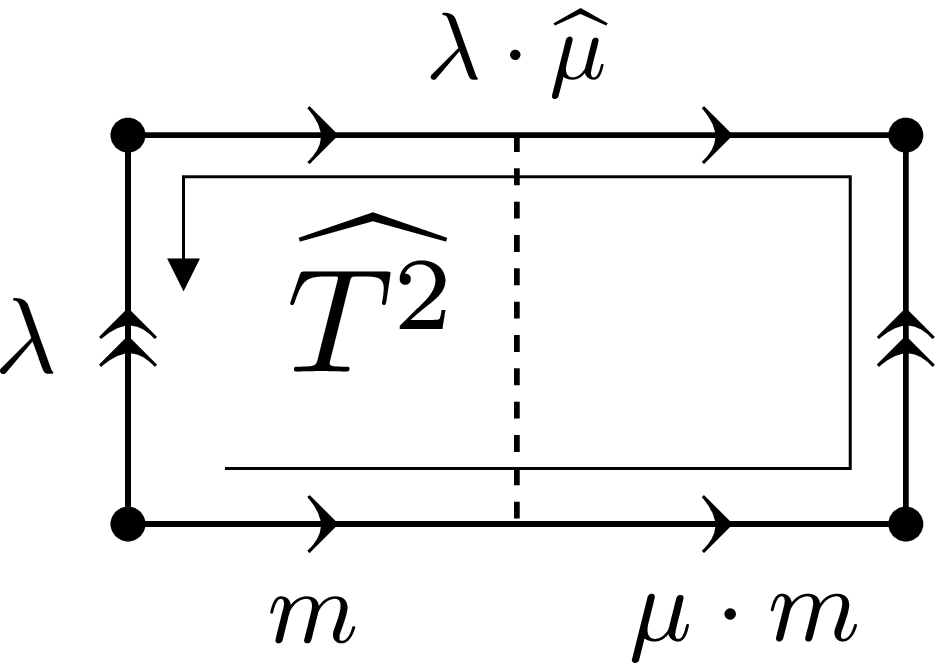}
  \end{center}
  \caption{The symbol $m$ denotes
    a lift of the meridian $\mu$ and a lift of longitude $\lambda$
    is denoted by the same symbol.}
  \label{fig:double_cover_torus}
\end{figure}

Note that the chains corresponding to $\widehat{\mu}$ and $\widehat{T^2}$ 
is expressed as
$m + \mu \cdot m$ 
and $T^2 + \mu \cdot T^2$
tensored with a vector of $\sll$
in $\sll \otimes_{\Z[\pi_1(\widehat{T^2})]} C_*(\univcover{T^2};\Z)$ where $\univcover{T^2}$ is
the universal cover of $T^2 = \boundaryTorus$.
We will denote lifts of chains briefly by the same symbols when no confusion can arise.

The homology groups of $C_*(\widehat{T^2};\sll)$ and $C_*(D^2 \times S^1;\sll)$ 
are expressed as follows.
\begin{lemma}
  \label{lemma:homology_doubleCoverTorus_solidTorus}
  \begin{align*}
    H_*(\widehat{T^2};\sll) &=
    \begin{cases}
      \sll \otimes \widehat{T^2}
      & (*=2) \\
      \sll \otimes \widehat{\mu} \oplus \sll \otimes \lambda
      & (*=1) \\
      \sll \otimes \bpt & (*=0) \\
      \bm{0} & {\rm otherwise},
    \end{cases}
    \\
    H_*(D^2 \times S^1;\sll) &=
    \begin{cases}
      \sll \otimes \lambda & (*=1) \\
      \sll \otimes \bpt & (*=0) \\
      \bm{0} & {\rm otherwise}.
    \end{cases}
  \end{align*}
\end{lemma}

The short exact sequence~\eqref{eqn:short_exact_sll} induces the
Mayer--Vietoris long exact sequence of homology groups with $\sll$-coefficient:
\begin{equation}
  \label{eqn:Mayer_Vietoris}
  \cdots \to H_{i+1}(\Sigma)
  \to H_i(\widehat{T^2})
  \to H_i(\coverExterior) \oplus H_i(D^2 \times S^1)
  \to H_i(\Sigma_2)
  \to \cdots
  \tag{$\mathcal{H}$}
\end{equation}

We can deduce the twisted homology groups $H_*(\Sigma_2;\sll)$ and
$H_*(\coverExterior;\sll)$
from the Mayer--Vietoris exact sequence~\eqref{eqn:Mayer_Vietoris}.
\begin{proposition}
  \label{prop:twistedhomology_coveringspaces}
  Let $\rho$ be an irreducible metabelian $\SL$-representation of $\knotgroup$
  and $\rho'$ another one as in Lemma~\ref{lemma:adjoint_metabelian}.
  If the twisted chain complex $C_*(\Sigma_2;\C_{\xi})$ is acyclic,
  then the twisted homology groups of $C_*(\Sigma_2;\sll)$ and $C_*(\coverExterior;\sll)$
  is expressed as follows:
  \begin{align*}
    H_*(\Sigma_2;\sll) &\simeq
    \begin{cases}
      V_1 & (*=0, 3) \\
      \bm{0} & {\rm otherwise},
    \end{cases}&
    H_*(\coverExterior;\sll) &\simeq
    \begin{cases}
      V_2 & (*=2) \\
      \sll & (*=1) \\
      V_1 & (*=0) \\
      \bm{0} & {\rm otherwise}.
    \end{cases}
  \end{align*}
  where $V_1$ and $V_2$ denote the $1$-dimensional subspace and $2$-dimensional one of $\sll$
  corresponding to the representation $Ad \circ \rho' = \psi_1 \oplus \psi_2$ of $\knotgroup$.
\end{proposition}

\begin{proof}
  Lemma~\ref{lemma:decomp_twistedchain_DB} gives the decomposition
  $C_*(\Sigma_2;\sll)=C_*(\Sigma_2;\C) \oplus C_*(\Sigma_2;\C_\xi) \oplus C_*(\Sigma_2;\C_{\bar \xi})$
  where the coefficient of $C_*(\Sigma_2;\C)$ corresponds to the subspace $V_1$ in $\sll$.
  By definition, the twisted chain complex
  $C_*(\Sigma_2;\C_{\bar \xi})$ is also acyclic when
  $C_*(\Sigma_2;\C_\xi)$ is acyclic.
  Hence $H_*(\Sigma_2;\sll)$ coincides with $H_1(\Sigma_2;\C)$ and it is expressed as in our statement.
  The homology group $H_*(\coverExterior;\sll)$ follows from
  the Mayer-Vietoris exact sequence~\eqref{eqn:Mayer_Vietoris} and
  Lemma~\ref{lemma:homology_doubleCoverTorus_solidTorus}.
\end{proof}

\begin{remark}
  In Proposition~\ref{prop:twistedhomology_coveringspaces},
  the basis of $H_*(\Sigma_2;\sll)$ is given by
  the fundamental cycle and the base point of $\Sigma_2$.
  Also the generators of $H_*(\coverExterior;\sll)$ are given by
  the subset of basis of $H_*(\widehat{T^2};\sll)$,
  \ie
  these generators are expressed as the chains
  $\widehat{T^2}$, $\widehat{\mu}$ and $\bpt$
  tensored with bases of $V_2$, $\sll$ and $V_1$.
\end{remark}

We also give the inverse of Proposition~\ref{prop:twistedhomology_coveringspaces} needed later.
\begin{proposition}
  \label{prop:basis_twistedhomology_coverExterior}
  Under the same notations of Proposition~\ref{prop:twistedhomology_coveringspaces}, 
  the twisted chain complex $C_*(\Sigma_2;\C_\xi)$ is acyclic if
  the twisted homology group $H_*(\coverExterior;\sll)$ is generated as follows:
  $$
  H_*(\coverExterior;\sll) \simeq
  \begin{cases}
    V_2 \otimes \widehat{T^2} & (*=2) \\
    \sll \otimes \widehat{\mu} & (*=1) \\
    V_1 \otimes \bpt & (*=0) \\
    \bm{0} & otherwise.
  \end{cases}
  $$  
\end{proposition}
\begin{proof}
  It also follows from the Mayer--Vietoris exact sequence~\eqref{eqn:Mayer_Vietoris}
  and Lemma~\ref{lemma:homology_doubleCoverTorus_solidTorus}.
\end{proof}

We can derive the Reidemeister torsion of $\Sigma_2$ with the representation $\I \oplus (\xi \oplus \xi^{-1})$
from the short exact sequence~\eqref{eqn:short_exact_sll} and
bases of homology group as in Lemma~\ref{lemma:homology_doubleCoverTorus_solidTorus}.

\begin{proposition}
  \label{prop:Torsion_Sigma2}
  Let $\rho$ be an irreducible metabelian $\SL$-representation of $\knotgroup$
  and $\rho'$ another one as in Lemma~\ref{lemma:adjoint_metabelian}.
  We suppose that 
  the twisted chain complex $C_*(\Sigma_2;\xi)$ is acyclic.
  Then the absolute value of Reidemeister torsion for $\Sigma_2$ and $\xi$ is 
  expressed as 
  $$ |\Tor{\Sigma_2}{\xi}|^2
  = \frac{1}{|H_1(\Sigma_2;\Z)|} \cdot
  |\Tor{\coverExterior}{\sll} \cdot \TorMV{\hbox{\ref{eqn:Mayer_Vietoris}}}|
  $$
  where the Mayer--Vietoris exact sequence~\eqref{eqn:Mayer_Vietoris}
  is equipped with bases of twisted homology groups $H_*(\widehat{T^2};\sll)$ and $H_*(\coverExterior;\sll)$
  as in Lemma~\ref{lemma:homology_doubleCoverTorus_solidTorus}.
\end{proposition}
\begin{proof}
  Applying the Multiplicativity property of Reidemeister torsion
  (we refer to~\cite[Proposition~$0.11$]{Porti:1997} and \cite[Proposition~$2.4.4$]{YY1})
  to the short exact sequence~\eqref{eqn:short_exact_sll},
  we have the following equality of Reidemeister torsions up to a sign:
  \begin{equation}
    \label{eqn:MultiplicativePropTorsion}
    \Tor{\Sigma_2}{\sll}\Tor{\widehat{T^2}}{\sll}=
    \pm \Tor{\coverExterior}{\sll}\Tor{D^2 \times S^1}{\sll}\TorMV{\hbox{\ref{eqn:Mayer_Vietoris}}}
  \end{equation}
  Here the Reidemeister torsions for $\widehat{T^2}$, $\coverExterior$ and $D^2 \times S^1$ are defined
  by non--acyclic chain complexes.
  By direct computation,
  one can show the Reidemeister torsions of $\widehat{T^2}$ and $D^2 \times S^1$ are equal to $\pm 1$
  for bases of $H_*(\widehat{T^2};\sll)$ and $H_*(\coverExterior;\sll)$ as in
  Lemma~\ref{lemma:homology_doubleCoverTorus_solidTorus}.
  
  Next the Reidemeister torsion $\Tor{\Sigma_2}{\sll}$ turns into the product
  $$
  \Tor{\Sigma_2}{\C}\Tor{\Sigma_2}{\C_\xi}\Tor{\Sigma_2}{\C_{\bar \xi}}
  $$
  by the decomposition given in Lemma~\ref{lemma:decomp_twistedchain_DB}. 
  Moreover the first factor $\Tor{\Sigma_2}{\C}$ turns out to be the order $\pm |H_1(\Sigma_2;\Z)|$
  by \cite[Proposition~$3.10$]{Porti:MayberryMurasugi} and
  Proposition~\ref{prop:twistedhomology_coveringspaces}.
  The last factor $\Tor{\Sigma_2}{\C_{\bar \xi}}$ coincides with the complex conjugate of
  $\Tor{\Sigma_2}{\C_{\xi}}$ from the definition.
  Combining these results, we can rewrite the equality~\eqref{eqn:MultiplicativePropTorsion} as
  $$
  |H_1(\Sigma_2;\Z)| \cdot |\Tor{\Sigma_2}{\C_\xi}|^2 =
  \pm \Tor{\coverExterior}{\sll} \cdot \TorMV{\hbox{\ref{eqn:Mayer_Vietoris}}}.
  $$
\end{proof}

\begin{remark}
  The Reidemeister torsion $\Tor{\coverExterior}{\sll}$ and
  $\TorMV{\hbox{\ref{eqn:Mayer_Vietoris}}}$ depend on the choice of
  basis of $H_*(\coverExterior;\sll)$.
  From the last equality of the proof of Proposition~\ref{prop:Torsion_Sigma2},
  it follows that the product is
  independent of the choice of basis of the twisted homology group
  $H_*(\coverExterior;\sll)$ and
  gives the topological invariant of $\Sigma_2$.
\end{remark}

In particular, when we choose the basis of $H_*(\coverExterior;\sll)$
as in Proposition~\ref{prop:basis_twistedhomology_coverExterior},
the Reidemeister torsion $\TorMV{\hbox{\ref{eqn:Mayer_Vietoris}}}$ turns into $\pm 1$.
However we will consider 
$\Tor{\coverExterior}{\sll}$ and $\TorMV{\hbox{\ref{eqn:Mayer_Vietoris}}}$
for another basis of $H_1(\coverExterior;\sll)$,
which is given by the lift of a preferred longitude.

\subsection{Torsion of cyclic covers over knot exteriors}
\label{subsect:Torsion_cyclic_covers_knotexteriors}
The purpose of this subsection is to compute the torsion $\Tor{\coverExterior}{\sll}$
using the twisted Alexander polynomial of $\knotexterior$.
We need to observe the twisted homology group of $\coverExterior$
to apply such a method, which was developed in~\cite{YY1, DuboisYamaguchi:derivativeFormula}.

Roughly speaking, the Reidemeister torsion for the non--acyclic twisted chain complex 
$C_*(\coverExterior;\sll)$ agrees with 
the torsion of the quotient $C_*(\coverExterior;\sll) / C'_*$
where $C'_*$ is a subchain complex
which is isomorphic to the twisted homology group $H_*(\coverExterior;\sll)$.
When this decomposition 
$C_*(\coverExterior;\sll) = C'_* \oplus C_*(\coverExterior;\sll) / C'_*$
also defines chain complexes in the coefficient $\C(t) \otimes \sll$,
we can express the torsion of the quotient complex as 
the rational function given by the twisted Alexander polynomial for $\coverExterior$
and a polynomial corresponding to $C'_*$ with the coefficient $\C(t) \otimes \sll$
and recover the torsion of $C_*(\coverExterior;\sll) / C'_*$,
which coincides with the torsion of $C_*(\coverExterior;\sll)$,
by evaluating this rational function at $t=1$.

To do this, we need to choose a suitable subchain complex $C'_*$ in $C_*(\coverExterior;\sll)$.
This problem is related to the possibility on the choice of basis of $H_*(\coverExterior;\sll)$.
We begin with a decomposition of $C_*(\coverExterior;\sll)$ to observe
 $H_*(\coverExterior;\sll)$ in detail.

We denote by $G$ the covering transformation group of $\coverExterior$, 
\ie
$G=\langle g \,|\, g^2 =1 \rangle$.
By the similar way of \cite[Lemma~$4.3$ and Eq.~$4$--$8$]{DuboisYamaguchi:abelianCovering}, 
we can expressed the twisted chain complex $C_*(\coverExterior;\sll)$ as the direct sum of 
those of the base space $\knotexterior$.
\begin{lemma}
  \label{lemma:decomp_twistedcomplex_coverExterior}
  Let $f_{\pm 1}$ denote $(1 \pm g)/2$ in the group ring $\C[G]$.
  Then the twisted chain complex $C_*(\coverExterior;\sll)$ is decomposed as
  \begin{equation}
    \label{eqn:decompo_chain_cpx_coverExterior}
    C_*(\coverExterior;\sll) \simeq
    C_*(\knotexterior; \sll \otimes \C[f_1]) \oplus C_*(\knotexterior; \sll \otimes \C[f_{-1}])
  \end{equation}
  where $C_*(\knotexterior; \sll \otimes \C[f_{\pm 1}])$ is defined by
  $(\sll \otimes_{\C} \C[f_{\pm 1}]) \otimes_{\Z[\knotgroup]} C_*(\univcover{\knotexterior;\Z})$
  via $Ad \circ \rho'$ and the projection $\knotgroup \to \knotgroup / \liftknotgroup = G$.

  Moreover we can identify the chain complex $C_*(\knotexterior; \sll \otimes \C[f_1])$
  with the twisted chain complex $C_*(\knotexterior;\sll)$ defined by $Ad \circ \rho'$
  and $C_*(\knotexterior; \sll \otimes \C[f_{-1}])$
  with the twisted chain complex of $\knotexterior$ by $(-1)^{[\,\cdot\,]} \cdot Ad \circ \rho'$.
\end{lemma}
\begin{proof}
  It follows from the same way of~\cite[Lemma~$4.3$]{DuboisYamaguchi:abelianCovering} that 
  $$
  C_*(\coverExterior;\sll)
  \simeq (\sll \otimes_{\C} \C[G]) \otimes_{\Z[\knotgroup]} C_*(\univcover{\knotexterior};\Z)  
  $$
  where we take the tensor product $\otimes_{\Z[\knotgroup]}$ in
  the right hand side under the adjoint representation $Ad \circ
  \rho'$ and the projection $\knotgroup \to G = \knotgroup /
  \liftknotgroup$.  This isomorphism is given by $\Phi(x \otimes
  c) = (x \otimes 1 )\otimes c$.
  We can think of $\C[G]$ as a $2$-dimensional vector space and vectors $\{f_1, f_{-1}\}$ as a basis of $\C[G]$.
  Together with $\C[G] = \C[f_1] \oplus \C[f_{-1}]$,
  we have the decomposition of $C_*(\coverExterior;\sll)$ in our statement.

  The element $g$ of $G$ acts on $\C[G]$ linearly. The vectors $f_{\pm 1}$ are the eigenvector for the eigenvalues
  $\pm 1$ of this action.
  Therefore we can regard the chain complexes
  $(\sll \otimes \C[f_{\pm 1}]) \otimes_{\Z[\knotgroup]} C_*(\univcover{\knotexterior;\Z})$
  as the twisted chain complex of $\knotexterior$ with $\sll \otimes \C$ the coefficient defined by
  the representations $(Ad \otimes \rho') \otimes (\pm 1)^{[\,\cdot\,]}$.
  This gives the identification between
  $(\sll \otimes \C[f_{\pm 1}]) \otimes_{\Z[\knotgroup]} C_*(\univcover{\knotexterior;\Z})$
  and the twisted chain complexes of $\knotexterior$ defined by
  $Ad \circ \rho'$ and $(-1)^{[\,\cdot\,]} (Ad \circ \rho')$.
\end{proof}
\begin{proposition}
  \label{prop:decomp_twistedhomology_coverExterior}
  The twisted homology $H_*(\coverExterior;\sll)$ has the following decomposition:
  \begin{align}
    H_*(\coverExterior;\sll)
    &\simeq
    H_*(\knotexterior;\sll) \oplus H_*(\knotexterior;\sll \otimes \C_{-1})
    \label{eqn:isom_twistedhomology}\\
    &\simeq
    H_*(\knotexterior;\C) \oplus H_*(\knotexterior;\sll)^{\oplus 2}
    \label{eqn:isom_twistedhomology_knotexterior}
  \end{align}
  where $\C_{-1}$ denotes the $1$-dimensional representation $(-1)^{[\, \cdot \,]}$ of $\knotgroup$.
  Moreover the twisted homology group $H_*(\knotexterior;\sll)$ is isomorphic to $H_*(\knotexterior;V_2)$
  whose coefficient $V_2$ is the $2$-dimensional vector space for the summand $\psi_2$ in $Ad \circ \rho'$.
\end{proposition}
\begin{proof}
  The first isomorphism follows from Lemma~\ref{lemma:decomp_twistedcomplex_coverExterior}.
  The second summand in the right hand side of~\eqref{eqn:isom_twistedhomology}
  is determined by the twisted chain complex corresponding to
  the representation
  $(-1)^{[\, \cdot \,]} (Ad \circ \rho') = \I \oplus (\sqrtrep \cdot (-1)^{[\, \cdot\,]} \rho)$.
  The $\SL$-representation $(-1)^{[\, \cdot\,]} \rho$ is conjugate to itself $\rho$
  by Proposition~\ref{prop:fixed_points_involution}.
  Thus $(-1)^{[\, \cdot \,]} (Ad \circ \rho')$ is conjugate to
  $\I \oplus \sqrtrep \rho = \I \oplus \psi_2$.
  This conjugation between representations induces the isomorphism from 
  twisted homology group $H_*(\knotexterior;\sll \otimes \C_{-1})$ to
  $H_*(\knotexterior;\C) \oplus H_*(\knotexterior; V_2)$.

  We need to show the isomorphism between $H_*(\knotexterior)$ and $H_*(\knotexterior;V_2)$
  for the isomorphism~\eqref{eqn:isom_twistedhomology_knotexterior}.
  The equality $Ad \circ \rho' = \signrep \oplus \sqrtrep \rho$ gives the isomorphism:
  $$
  H_*(\knotexterior;\sll) \simeq H_*(\knotexterior;V_1) \oplus H_*(\knotexterior;V_2)
  $$
  and the coefficient $V_1$ is the vector space $\C_{-1}$.
  By direct calculation, we can see that the homology group $H_*(\knotexterior;\C_{-1}) =\bm{0}$ which proves the proposition.
\end{proof}

Proposition~\ref{prop:decomp_twistedhomology_coverExterior} shows that 
the twisted homology group
$H_*(\coverExterior;\sll)$ is determined by $H_*(\knotexterior;\sll)$.
The possibilities on the choices of basis for $H_*(\coverExterior;\sll)$
is also determined by the possibilities for $H_*(\knotexterior;\sll)$.
We need to consider which cycles gives a basis of $H_1(\knotexterior;\sll)$.

\medskip

We focus on cycles given by the preferred longitude in
$C_1(\knotexterior;\sll)$ for the problem related to the choices of
basis for $H_1(\coverExterior;\sll)$.  This means that we assume that
an irreducible metabelian $\SL$-representation $\rho'$ is
$\lambda$-regular where $\lambda$ is the preferred longitude on
$\boundaryTorus$.

To describe the basis of $H_*(\knotexterior;\sll)$, 
we will use the following notations about eigenvectors of $\rho(\mu)$, which also give
eigenvectors of $Ad \circ \rho' (\mu)$.
\begin{definition}
  \label{def:egenvectors}
  Let $\rho$ and $\rho'$ denote irreducible metabelian $\SL$-representations of $\knotgroup$
  which satisfy that $Ad \circ \rho' = \signrep \oplus \sqrtrep \rho$ on the ordered basis
  $\{H, E, F\}$ of $\sll$ as in Lemma~\ref{lemma:adjoint_metabelian}.
  For the matrix $\rho(\mu) =
  \begin{pmatrix}
    0 & \sqrt{-1} \\
    \sqrt{-1} & 0
  \end{pmatrix}$ acting on ${}_\C \langle E, F \rangle$,
  we denote by $P^\rho$ and $Q^\rho$ eigenvectors $\frac{1}{\sqrt{2}}(E-F)$ and $\frac{1}{\sqrt{2}}(E+F)$
  for the eigenvalues $\mp \sqrt{-1}$.
\end{definition}
\begin{remark}
  The vectors $P^\rho$ and $Q^{\rho}$ are eigenvectors of $Ad_{\rho'(\mu)}$ for the eigenvalues $\pm 1$.
\end{remark}
\begin{proposition}
  \label{prop:lambdabasis_twistedhomology}
  Under the assumption that $\rho'$ is $\lambda$-regular,
  we can choose the following basis of $H_*(\coverExterior;\sll)$:
  \begin{equation}
    \label{eqn:basis_twistedhomology_coverExterior}
    H_*(\coverExterior;\sll) \simeq 
    \begin{cases}
      {}_{\C}\langle
      P^{\rho} \otimes \widehat{T^2},\,
      Q^{\rho} \otimes \widehat{T^2}
      \rangle & (*=2) \\
      {}_{\C}\langle
      H \otimes \widehat{\mu},\,
      P^{\rho} \otimes \lambda,\,
      Q^{\rho} \otimes \lambda
      \rangle & (*=1) \\
      {}_{\C}\langle
      H \otimes \bpt
      \rangle & (*=0)
    \end{cases}
  \end{equation}
  where $H$ generates the eigenspace $V_1$ and
  $P^\rho$ and $Q^\rho$ are the eigenvectors as in Definition~\ref{def:egenvectors}.
\end{proposition}
\begin{proof}
  By Proposition~\ref{prop:decomp_twistedhomology_coverExterior},
  we have the isomorphism
  $$
  H_*(\coverExterior;\sll) \simeq
  H_*(\knotexterior;\sll) \oplus H_*(\knotexterior;\sll \otimes \C_{-1})
  $$
  and the decomposition 
  $H_*(\knotexterior;\sll \otimes \C_{-1}) \simeq
  H_*(\knotexterior;\C) \oplus H_*(\knotexterior;V_2)$.
  We prove the cycles as in Eq.~\eqref{eqn:basis_twistedhomology_coverExterior} gives
  bases of twisted homology groups $H_*(\knotexterior;\sll)$ and
  $H_*(\knotexterior;\sll \otimes \C_{-1})$.

  We consider the image of chains under the isomorphism $\Phi$
  from $C_*(\coverExterior;\sll)$ to the direct sum of
  $C_*(\knotexterior;\sll \otimes \C[f_{\pm 1}])$ as in Eq.~\eqref{eqn:decompo_chain_cpx_coverExterior}.
  Under this decomposition, the image $P^\rho \otimes \widehat{T^2}$ is expressed as
  \begin{align*}
    \Phi(P^\rho \otimes \widehat{T^2})
    &=
    P^\rho \otimes f_1 \otimes (1+\mu) T^2 +
    P^\rho \otimes f_{-1} \otimes (1+\mu) T^2 \\
    &=
    2 P^\rho \otimes f_1 \otimes T^2.
  \end{align*}
  We also have the equality that
  $\Phi(Q^\rho \otimes \widehat{T^2}) = 2 Q^\rho \otimes f_1 \otimes T^2$.
  It follows from the $\lambda$-regularity of $\rho'$ that 
  these images of chains $P^\rho \otimes \widehat{T^2}$ and $Q^\rho \otimes \widehat{T^2}$ gives generators of
  the homology groups $H_2(\knotexterior;\sll)$ and $H_2(\knotexterior;V_2)$.
  Similarly we can show the images of $P^\rho \otimes \lambda$ and $Q^\rho \otimes \lambda$
  turns into $2P^\rho \otimes \lambda$ and $2Q^\rho \otimes \lambda$ which gives non--trivial
  homology classes.
  The the image $\Phi(H \otimes \widehat{\mu})$ turns out to be $2 H \otimes f_{-1} \otimes \mu$
  since $H$ is also an eigenvector for the eigenvalue $-1$ of $Ad \circ \rho(\mu)$.
  The cycle $H \otimes f_{-1} \otimes \mu$ gives a generator of $H_1(\knotexterior;\C)$
  in Eq.~\eqref{eqn:isom_twistedhomology_knotexterior}.
  Last it follows from the isomorphism in Proposition~\ref{prop:decomp_twistedhomology_coverExterior}
  that $H \otimes \bpt$ gives a generator of $H_0(\knotexterior;\C) \simeq H_0(\coverExterior;\sll)$.
\end{proof}

We now proceed to construct a subchain complex of $C_*(\coverExterior;\sll)$ which arises from
the basis in Proposition~\ref{prop:lambdabasis_twistedhomology}.
\begin{definition}
  \label{def:subchain_complex}
  Suppose that an irreducible metabelian representation $\rho'$ is $\lambda$-regular.
  We define the subchain complex $C'_*$ of $C_*(\coverExterior;\sll)$ as 
  \begin{gather*}
    0 \to C'_2 \to C'_1 \to C'_0 \to 0 \\
    C'_2 = 
    {}_{\C}\langle
    P^{\rho} \otimes \widehat{T^2},\,
    Q^{\rho} \otimes \widehat{T^2}
    \rangle, \quad 
    C'_1 =
    {}_{\C}\langle
    H \otimes \widehat{\mu},\,
    P^{\rho} \otimes \lambda,\,
    Q^{\rho} \otimes \lambda
    \rangle, \quad
    C'_0 =
    {}_{\C}\langle
    H \otimes \bpt
    \rangle.
  \end{gather*}
\end{definition}
Note that the restriction of boundary operators are $0$-homomorphism and 
the homology group $H_*(C'_*)$ coincides with $H_*(\coverExterior;\sll)$.
We denote by $C''_*$ the chain complex defined as
the quotient of $C_*(\coverExterior;\sll)$ by $C'_*$.
It follows that $C''_*$ is acyclic which is due to the induced homology long exact sequence 
from the short exact sequence:
\begin{equation}
  \label{eqn:short_exact_quotient}
  0 \to C'_* \to C_*(\coverExterior;\sll) \to C''_* \to 0.
\end{equation}

The Reidemeister torsion of the acyclic chain complex $C''_*$ coincides with 
the Reidemeister torsion of the non--acyclic one $C_*(\coverExterior;\sll)$.
\begin{proposition}
  \label{prop:eqn_torsions}
  We assume that $\rho'$ is $\lambda$-regular and choose a basis of
  $H_*(\coverExterior;\sll)$ as in Proposition~\ref{prop:decomp_twistedhomology_coverExterior}.
  Then we have the following equality:
  $$
  \TorCpx{C''_*} = \pm \Tor{\coverExterior}{\sll}.
  $$
\end{proposition}
\begin{proof}
  Applying the Multiplicativity property of Reidemeister torsion to 
  the short exact sequence~\eqref{eqn:short_exact_quotient},
  we have the equality:
  \begin{equation}
    \label{eqn:MP_quotient}
    \TorCpx{C''_*}\TorCpx{C'_*}
    = \pm \Tor{\coverExterior}{\sll}
    \TorCpx{\mathcal{H}(C'_*, C_*(\coverExterior;\sll), C''_*)}
  \end{equation}
  where $\mathcal{H}(C'_*, C_*(\coverExterior;\sll), C''_*)$ denotes
  the induced long exact sequence of homology groups.
  Since $H_*(C'_*) = C'_*$, the torsion $\TorCpx{C'_*} = \pm 1$ by definition.
  We have chosen the basis of $H_*(\coverExterior;\sll)$ as the same one of 
  $H_*(C'_*)$. Hence the torsion for the exact sequence
  $\mathcal{H}(C'_*, C_*(\coverExterior;\sll), C''_*)$
  also turns out to be $\pm 1$.
  We can rewrite Eq.~\eqref{eqn:MP_quotient} as that of the proposition.
\end{proof}

We can also define the following subchain complex in the twisted chain complex 
with the variable $t$ for $\coverExterior$.

\begin{definition}
  Under the assumption that $\rho'$ is $\lambda$-regular, we denote by $C'_*(t)$ 
  the subchain complex of $C_*(\coverExterior;\sllt)$ defined by
  \begin{align*}
    C'_2(t) &= 
    {}_{\C(t)}\langle
    1 \otimes P^{\rho} \otimes \widehat{T^2},\,
    1 \otimes Q^{\rho} \otimes \widehat{T^2}
    \rangle, \\
    C'_1(t) &=
    {}_{\C(t)}\langle
    1 \otimes H \otimes \widehat{\mu},\,
    1 \otimes P^{\rho} \otimes \lambda,\,
    1 \otimes Q^{\rho} \otimes \lambda
    \rangle, \\
    C'_0(t) &=
    {}_{\C(t)}\langle
    1 \otimes H \otimes \bpt
    \rangle
  \end{align*}
  and the boundary operators are given by 
  \begin{gather*}
    0 \to C'_2(t) \xrightarrow{\partial'_2} C'_1(t) \xrightarrow{\partial'_2} C'_0(t) \to 0 \\
    \partial'_2 =
    \begin{pmatrix}
      0 & 0 \\
      t^2-1 & 0 \\
      0 & t^2-1
    \end{pmatrix}, \quad 
    \partial'_1 =
    \begin{pmatrix}
      t^2-1 & 0 & 0 \\
    \end{pmatrix}.
  \end{gather*}
  We also denote by $C''_*(t)$ the quotient $C_*(\coverExterior;\sllt)$ by $C'_*(t)$.
\end{definition}

To define subchain complex $C'_*(t)$, we need a closed loop on $\partial \coverExterior$ 
whose homology class is trivial, \ie which is included in $\ker p^*\alpha$.
This is a reason to choose cycles given by the longitude, 
corresponding to the generators of $H_1(\coverExterior;\sll)$
as in the isomorphism~\eqref{eqn:basis_twistedhomology_coverExterior}.

We can recover the torsion of $C_*(\coverExterior;\sll)$
by substituting $t=1$ into the torsion of $C''_*(t)$,
as discussed below.
Moreover the torsion of $C''_*(t)$ is determined by that of $C_*(\coverExterior;\sllt)$.
According to the definition of the twisted Alexander polynomial by M.~Wada~\cite{Wada94} 
and interpretation as Reidemeister torsion by P.~Kirk and C.~Livingston~\cite{KL}, 
we regard the Reidemeister torsion of the chain complex $C_*(\coverExterior;\sllt)$ 
as the twisted Alexander polynomial for
$\coverExterior$ and the pull--back $Ad \circ \rho'$ and $\alpha$.
This viewpoint gives a computation method of $\Tor{\coverExterior}{\sll}$ 
by evaluation of the twisted Alexander polynomial.

\begin{proposition}
  \label{prop:limit_twistedAlex}
  If $\rho'$ is $\lambda$-regular, then we can express
  $\Tor{\coverExterior}{\sll}$ with the basis of $H_*(\coverExterior;\sll)$
  as in~\eqref{eqn:basis_twistedhomology_coverExterior} as
  $$
  \Tor{\coverExterior}{\sll}
  = \pm \lim_{t \to 1}
  \frac{\twistedAlexAd{\coverExterior}{\rho'}}{t^2-1}.
  $$
\end{proposition}

The proof of Proposition~\ref{prop:limit_twistedAlex} will be divide into two part. 
First we need to prove the torsion of $C''_*(t)$ is given by
$\twistedAlexAd{\coverExterior}{\rho'} / (t^2-1)$.
In the second step, evaluating this rational function at $t=1$
we see the torsion $\TorCpx{C''_*(t)}$ turns out to be
$\TorCpx{C''_*} = \pm \Tor{\coverExterior}{\sll}$ under the isomorphism in
Proposition~\ref{prop:eqn_torsions}.
The first step is divided into three lemmas.
\begin{lemma}
  \label{lemma:Torsion_C'_variable}
  The subchain complex $C'_*(t)$ is an acyclic chain complex.
  The torsion of this chain complex is equal to $\pm (t^2 -1)$.
\end{lemma}
\begin{proof}
  This lemma follows from the construction.
\end{proof}

\begin{lemma}
  \label{lemma:Rtorsion_coverExterior_variable}
  If $\rho'$ is $\lambda$-regular,
  then the twisted chain complex $C_*(\coverExterior;\sllt)$ defined by $p^* (Ad \circ \rho')$
  is acyclic.
  Moreover the torsion of $C_*(\coverExterior;\sllt)$ is given by the twisted Alexander polynomial
  $\twistedAlexAd{\coverExterior}{\rho'}$.
\end{lemma}
\begin{proof}
  By~\cite[Lemma~$4.3$]{DuboisYamaguchi:abelianCovering},
  we have the following decomposition of $C_*(\coverExterior;\sllt)$:
  \begin{align}
    C_*(\coverExterior;\sllt)
    &\simeq (\C(t) \otimes \sll \otimes \C[G]) \otimes_{\Z[\knotgroup]} C_*(\univcover{\knotexterior};\Z) \notag \\
    &= (\C(t) \otimes \sll \otimes \C[f_1]) \otimes_{\Z[\knotgroup]} C_*(\univcover{\knotexterior};\Z) \notag \\
    & \quad \oplus (\C(t) \otimes \sll \otimes \C[f_{-1}]) \otimes_{\Z[\knotgroup]} C_*(\univcover{\knotexterior};\Z)
    \label{eqn:decompo_twistedchain_with_variable}
  \end{align}
  where we take tensor product under the representations $\alpha$, $Ad \circ \rho'$ and the projection
  $\knotgroup \to G$.
  The right hand side of Eq.~\eqref{eqn:decompo_twistedchain_with_variable} is isomorphic to the direct sum:
  $$
  C_*(\knotexterior;\sllt) \oplus C_*(\knotexterior; \C(t) \otimes (\sll \otimes \C_{-1})).
  $$
  Since we assume that $\rho'$ is $\lambda$-regular,
  the acyclicity of 
  $C_*(E_K;\sllt)$ follows from \cite[Proposition~$3.1.1$]{YY1}.
  The second summand $C_*(\knotexterior; \C(t) \otimes (\sll \otimes \C_{-1}))$ is defined by
  $\alpha \otimes \signrep (Ad \circ \rho').$
  We have seen that the representation $\signrep (Ad \circ \rho')$ is conjugate to
  $\I \oplus \sqrtrep \rho = \I \oplus \psi_2$
  in the proof of Proposition~\ref{prop:decomp_twistedhomology_coverExterior}.
  This conjugation of representation induces the following isomorphism between the homology groups:
  $$
  H_*(\knotexterior; \C(t) \otimes (\sll \otimes \C_{-1}))
  \simeq H_*(\knotexterior;\C(t)) \oplus H_*(\knotexterior;\C(t) \otimes V_2).
  $$
  It is known that $H_*(\knotexterior;\C(t))$ is trivial and
  Proposition~\ref{prop:decomp_twistedhomology_coverExterior} shows that
  $H_*(\knotexterior;\C(t) \otimes V_2)$ is isomorphic to $H_*(\knotexterior;\sllt)$, which is trivial.

  It remains to prove that the torsion of $C_*(\coverExterior;\sllt)$ coincides with
  the twisted Alexander polynomial. This follows from the result of~\cite[Section~$4$]{KL}. 
\end{proof}

\begin{lemma}
  \label{lemma:C_dprime_twistedAlex}
  Under the same hypothesis of Lemma~\ref{lemma:Rtorsion_coverExterior_variable},
  the chain complex $C''_*(t)$ is also acyclic and 
  its torsion is expressed as
  $$
  \TorCpx{C''_*(t)} =
  \frac{\twistedAlexAd{\coverExterior}{\rho'}}{t^2-1}
  $$
  up to a factor $\pm t^k$ $(k \in \Z)$.
\end{lemma}
\begin{proof}
  When we consider the homology long exact sequence induced from
  \begin{equation}
    \label{eqn:short_exact_variable}
    0 \to C'(t) \to C_*(\coverExterior;\sllt) \to C''_*(t) \to 0,
  \end{equation}
  we can assert that $C''_*(t)$ is also acyclic
  by Lemmas~\ref{lemma:Torsion_C'_variable} \&~\ref{lemma:Rtorsion_coverExterior_variable}.
  Applying Multiplicativity property for the short exact sequence~\eqref{eqn:short_exact_variable}
  of acyclic chain complexes, we have the following equality:
  $$
  \TorCpx{C'_*(t)} \TorCpx{C''_*(t)} = \Tor{\coverExterior}{\sllt}
  $$
  We obtain the proposition substituting torsions in 
  Lemma~\ref{lemma:Torsion_C'_variable} \&~\ref{lemma:Rtorsion_coverExterior_variable}
  into the above equality.
\end{proof}

We need to show the evaluation of $\TorCpx{C''_*(t)}$ gives the torsion $\Tor{\coverExterior}{\sll}$.
This is an application of~\cite[Proposition~3.3.1]{YY1} to our situation, which can be rewritten as follows.
\begin{lemma}[Proposition~3.3.1 in~\cite{YY1}]
  \label{lemma:limitformula}
  If $C_*(\coverExterior;\sllt)$ and $C'_*(t)$ is acyclic,  then the following relation holds:
  $$
  \lim_{t \to 1}
  \frac{
    \Tor{\coverExterior}{\sllt}
  }{
    \TorCpx{C'_*(t)}
  }
  = \pm \Tor{\coverExterior}{\sll}.
  $$
\end{lemma}
As an application of this lemma, we can show the proof of Proposition~\ref{prop:limit_twistedAlex}.
\begin{proof}[Proof of Proposition~\ref{prop:limit_twistedAlex}]
  It follows from Lemmas~\ref{lemma:Torsion_C'_variable}, \ref{lemma:Rtorsion_coverExterior_variable}
  and~\ref{lemma:limitformula}.
\end{proof}
In fact, Proposition~$3.3.1$ in \cite{YY1} includes
Proposition~\ref{prop:eqn_torsions} and Lemma~\ref{lemma:C_dprime_twistedAlex}.
They can make Proposition~$3.3.1$ in~\cite{YY1} easy to understand.
\begin{remark}
  The acyclic chain complexes $C_*(\coverExterior;\sllt)$ and $C'_*(t)$ correspond to the non--acyclic complexes
  $C_*(\coverExterior;\sll)$ and $C'_*$ in evaluating at $t=1$.
  However
  the acyclic chain complex $C''_*(t)$ corresponds to the acyclic one.
  It makes sense to take the evaluation of
  $\twistedAlexAd{\coverExterior}{\rho'} / (t^2-1)$
  as torsions of acyclic chain complexes.
\end{remark}

\subsection{Torsion of Mayer--Vietoris homology exact sequence}
\label{subsect:Torsion_Mayer_Vietoris_sequence}
We will express the torsion of the Mayer--Vietoris exact sequence~\eqref{eqn:Mayer_Vietoris} 
as the special value of rational function on the character varieties.
This is due to the identification between 
the twisted homology group $H_1(E_K;\sll)$ and the cotangent space of the $\SL$-character variety.
Roughly speaking, the rational function expresses the ratio of two $1$-forms on the character variety.
We will show that the special value of the rational function at an irreducible metabelian character
gives the torsion of the exact sequence~\eqref{eqn:Mayer_Vietoris}.

To observe the torsion of Mayer--Vietoris exact sequence~\eqref{eqn:Mayer_Vietoris},
we set the bases of each homology groups in the sequence.
We assume that $H_*(\Sigma_2; \C_{\xi})=\bm{0}$.
The coefficient vector space $\sll$ has the standard basis $\{E,\, H,\,  F\}$.
However we set a basis of $\sll$ as $\{H,\, P^\rho,\, Q^\rho\}$ where
$P^\rho$ and $Q^\rho$ are defined as $\frac{1}{\sqrt{2}}(E \mp F)$.
\begin{lemma}
  \label{lemma:determinant_MV}
  Suppose that $\sll$ is given by the basis $\{H,\,  P^\rho,\, Q^\rho\}$ and
  $\rho'$ is $\lambda$-regular.
  If we choose the bases of $H_*(\widehat{T^2};\sll)$, $H_*(D^2 \times S^1;\sll)$
  as in Lemma~\ref{lemma:homology_doubleCoverTorus_solidTorus} and the basis of
  $H_*(\coverExterior;\sll)$ as in Eq.~\eqref{eqn:basis_twistedhomology_coverExterior},
  then the torsion $\TorMV{\hbox{\ref{eqn:Mayer_Vietoris}}}$ is given by the determinant of
  the base change matrix
  $$
  T = \left(
  \{i_*(P^\rho \otimes \widehat{\mu}),\, i_*(Q^\rho \otimes \widehat{\mu})\}
  /  \{P^\rho \otimes \lambda,\, Q^\rho \otimes \lambda\}
  \right)
  $$
  in $H_1(\coverExterior;\sll)$,
  where $i_*$ denotes the induced homomorphism from $H_1(\widehat{T^2};\sll)$ to $H_1(\coverExterior;\sll)$
  by the inclusion $\widehat{T^2} \hookrightarrow \coverExterior$.
\end{lemma}
\begin{proof}
  It follows from the definition of torsion for an acyclic complex that
  the torsion $\TorMV{\hbox{\ref{eqn:Mayer_Vietoris}}}$ is equal to the determinant of the
  isomorphism:
  $$
  H_1(\widehat{T^2};\sll) \to
  H_1(\coverExterior;\sll) \oplus H_1(D^2 \times S^1;\sll),
  $$
  whose representation matrix is
  $
  \begin{pmatrix}
    1 & \bm{O} &  \bm{O} \\
    \bm{O} & T & \bm{O} \\
    \multicolumn{2}{c}{\I_3} & \I_3 
  \end{pmatrix}
  $
  where $\I_3$ denotes the $3 \times 3$ identity matrix.
\end{proof}

We need to observe where vectors $P^\rho \otimes \lambda$ and $Q^\rho \otimes \lambda$
live in the decomposition of the twisted homology group $H_1(\coverExterior;\sll)$ as in 
Proposition~\ref{prop:decomp_twistedhomology_coverExterior}:
\begin{align*}
  H_1(\coverExterior;\sll) 
  &\simeq
  H_1(\knotexterior;\sll) \oplus H_1(\knotexterior;\sll \otimes \C_{-1})\\
  &\simeq
  H_1(\knotexterior;\sll) \oplus H_1(\knotexterior;\C) \oplus H_*(\knotexterior;V_2 \otimes \C_{-1}).
\end{align*}
As what we have seen in the proof of Proposition~\ref{prop:lambdabasis_twistedhomology},
the vector $P^\rho \otimes \lambda$ is contained in the first summand $H_1(\knotexterior;\sll)$
and $Q^\rho \otimes \lambda$ is contained in
$H_*(\knotexterior;V_2 \otimes \C_{-1}) \simeq H_1(\knotexterior;\sll)$.
To compute the determinant of the base change matrix $T$,
it is enough to consider the ratio between two vector $P^\rho \otimes \widehat{\mu}$
(resp. $Q^\rho \otimes \widehat{\mu}$)
and $P^\rho \otimes \lambda$ (resp. $Q^\rho \otimes \lambda$)
in $H_*(\knotexterior;\sll)$.
To estimate these ratio,
we use {\it the formula of change of loops} given in~\cite[Proposition~$4.7$]{Porti:1997}.
We restate it for our situation.
\begin{lemma}[Proposition~$4.7$ in~\cite{Porti:1997}]
  \label{lemma:basechange_formula}
  Suppose that $\dim_{\C} H_1(\knotexterior;\sll) = 1$. 
  We assume that both of $P^\rho \otimes \widehat{\mu}$ and $P^\rho \otimes \lambda$ 
  give bases of $H_1(\knotexterior;\sll)$ and
  denote by $[P^\rho \otimes \widehat{\mu} / P^\rho \otimes \lambda]$
  the determinant of base change matrix from $P^\rho \otimes \lambda$
  to $P^\rho \otimes \widehat{\mu}$.
  
  Then the square of $[P^\rho \otimes \widehat{\mu} / P^\rho \otimes \lambda]$
  is expressed as the special value of the following rational function:
  \begin{equation}
    \label{eqn:square_det_basechange_matrix}
    [P^\rho \otimes \widehat{\mu} / P^\rho \otimes \lambda]^2 =
    \left.
    \frac{I_\lambda^2 -4}{I_{\widehat{\mu}}^2 -4}
    \left(\frac{d I_{\widehat{\mu}}}{dI_\lambda}\right)^2
    \right|_{\chi = \chi_{\rho'}}
  \end{equation}
  at the character of $\rho'$.
\end{lemma}

Combining the Lemmas~\ref{lemma:determinant_MV} \&~\ref{lemma:basechange_formula}, 
we can give the torsion $\TorMV{\hbox{\ref{eqn:Mayer_Vietoris}}}$ by
the special value of the rational function on the character variety.
\begin{proposition}
  \label{prop:torsion_MV_function}
  Under the assumption of Lemma~\ref{lemma:determinant_MV},
  we can express the torsion
  $\TorMV{\hbox{\ref{eqn:Mayer_Vietoris}}}$
  for the Mayer--Vietoris sequence~(\ref{eqn:Mayer_Vietoris})
  as
  $$
  \TorMV{\hbox{\ref{eqn:Mayer_Vietoris}}}
  =
  \left.
  \frac{I_\lambda^2 -4}{I_{\widehat{\mu}}^2 -4}
  \left(\frac{d I_{\widehat{\mu}}}{dI_\lambda}\right)^2
  \right|_{\chi = \chi_{\rho'}}.
  $$
\end{proposition}
\begin{proof}
  We have the isomorphism induced by the conjugation between $\signrep \rho$ and $\rho$:
  $$
  H_1(\knotexterior;V_2 \otimes \C_{-1}) \xrightarrow{\hbox{\small isom.}}
  H_1(\knotexterior;V_2) (\simeq H_1(\knotexterior;\sll))
  $$
  where we can identify $H_1(\knotexterior;V_2)$ with $H_1(\knotexterior;\sll)$ by
  Proposition~\ref{prop:decomp_twistedhomology_coverExterior}.
  Under this isomorphism, the vector $Q^\rho \otimes \widehat{\mu}$ (resp. $Q^\rho \otimes \lambda$)
  in $H_1(\knotexterior;V_2 \otimes \C_{-1})$ is sent to
  the vector $P^\rho \otimes \widehat{\mu}$ (resp. $P^\rho \otimes \lambda$)
  in $H_1(\knotexterior;\sll)$.
  Therefore the determinant of the base change matrix $T$ in Lemma~\ref{lemma:determinant_MV}
  turns into the square of determinant of base change matrix,
  given by Eq.~\eqref{eqn:square_det_basechange_matrix}.
\end{proof}

\subsection{Main theorem}
\label{subsect:Main_theorem}
From the results in the previous two SubSections~\ref{subsect:Torsion_cyclic_covers_knotexteriors}
\&~\ref{subsect:Torsion_Mayer_Vietoris_sequence},
we can rewrite the 
right hand side of the equality in Proposition~\ref{prop:Torsion_Sigma2}
as quantity determined by the knot exterior $\knotexterior$ and $\knotgroup$.

\begin{theorem}
  \label{thm:main_theorem}
  We suppose that $H_*(\Sigma_2;\C_\xi) = 0$ and the assumption of Lemma~\ref{lemma:determinant_MV}.
  Then we have the following equality:
  \begin{equation}
    \label{eqn:Main_Formula}
    |\Tor{\Sigma_2}{\xi}|^2 =
    \pm
    \left(
    \lim_{t \to 1} \frac{\twistedAlex[\sqrt{-1}\,t]{\knotexterior}{\rho}}{t^2 -1}
    \right)^2
    \left.
    \frac{I_\lambda^2 -4}{I_{\widehat{\mu}}^2 -4}
    \left(\frac{d I_{\widehat{\mu}}}{dI_\lambda}\right)^2
    \right|_{\chi = \chi_{\rho'}}.  
  \end{equation}
\end{theorem}
Theorem~\ref{thm:main_theorem} follows from the following three Lemmas 
which express the twisted Alexander polynomial
$\twistedAlexAd{\coverExterior}{\rho'}$
by that of the knot exterior $\knotexterior$.
\begin{lemma}[Theorem $4.1$ in~\cite{DuboisYamaguchi:abelianCovering}]
  \label{lemma:covering_formula}
  $$
  \twistedAlexAd{\coverExterior}{\rho'}
  = \twistedAlexAd{\knotexterior}{\rho'}
  \twistedAlexAd[-t]{\knotexterior}{\rho'}
  $$
\end{lemma}

\begin{lemma}[Theorem~$4.5$ in~\cite{yamaguchi:twistedAlexMeta}]
  \label{lemma:twistedAlex_Metabelian}
  \begin{equation}
    \label{eqn:twistedAlex_metabelian_adjoint}
    \twistedAlexAd{\knotexterior}{\rho'}
    = \pm (t-1)\Delta_K(-t) P(t)
  \end{equation}
  where $\Delta_K(t)$ is the Alexander polynomial of $K$ and
  the Laurent polynomial $P(t)$ satisfies that $P(t)=P(-t)$.
\end{lemma}

\begin{lemma}[Theorem~$4.7$ in~\cite{yamaguchi:twistedAlexMeta}]
  \label{lemma:polynomial_P}
  The Laurent polynomial $P(t)$ in
  Eq.~(\ref{eqn:twistedAlex_metabelian_adjoint})
  is also given by the twisted Alexander polynomial for the standard representation of $\rho$
  as
  $$
  P(t) =
  \frac{
    \twistedAlex[\sqrt{-1}\, t]{\knotexterior}{\rho}
  }{
    t^2 -1
  }
  $$
  up to a factor $\pm t^k \,(k \in \Z)$.
\end{lemma}

Note that it follows that $P(t)$ has only even degree terms from Lemma~\ref{lemma:polynomial_P}
and the result of C.~Herald, P.~Kirk and C.~Livingston~\cite{HeraldKirkLivingston2010} 
(we refer to~\cite[Remark~$4.8$]{yamaguchi:twistedAlexMeta} for the details).
\begin{proof}[Proof of Theorem~\ref{thm:main_theorem}]
  We start with the equality of Proposition~\ref{prop:Torsion_Sigma2}:
  $$
  |\Tor{\Sigma_2}{\C_\xi}|^2
  = \frac{\pm 1}{|H_1(\Sigma_2;\Z)|} \cdot
  \Tor{\coverExterior}{\sll} \cdot
  \TorMV{\hbox{\ref{eqn:Mayer_Vietoris}}}
  $$
  By Propositions~\ref{prop:limit_twistedAlex} \&~\ref{prop:torsion_MV_function},
  we can rewrite the above equality as follows:
  \begin{align*}
    |\Tor{\Sigma_2}{\C_\xi}|^2
    =
    \frac{\pm 1}{|H_1(\Sigma_2;\Z)|} \cdot
    \left(
    \lim_{t \to 1}  \frac{\twistedAlexAd{\coverExterior}{\rho'}}{t^2-1} 
    \right)\cdot
    \left.
    \frac{I_\lambda^2 -4}{I_{\widehat{\mu}}^2 -4}
    \left(\frac{d I_{\widehat{\mu}}}{dI_\lambda}\right)^2
    \right|_{\chi = \chi_{\rho'}}
  \end{align*}
  Substituting Lemmas~\ref{lemma:covering_formula} \&~\ref{lemma:twistedAlex_Metabelian},
  the right hand side turns into
  $$
  \frac{\pm 1}{|H_1(\Sigma_2;\Z)|} \cdot
  \left(
  \lim_{t \to 1}
  \Delta_K(t) \Delta_K(-t) P(t)P(-t)
  \right)\cdot
  \left.
  \frac{I_\lambda^2 -4}{I_{\widehat{\mu}}^2 -4}
  \left(\frac{d I_{\widehat{\mu}}}{dI_\lambda}\right)^2
  \right|_{\chi = \chi_{\rho'}}.
  $$
  Fox's formula shows that $\Delta_K(1)\Delta_K(-1) = \pm |H_1(\Sigma_2;\Z)|$ and
  the Laurent polynomial $P(t)$ satisfies $P(t)=P(-t)$.
  Applying Lemma~\ref{lemma:polynomial_P}, we obtain the desired equality:
  $$
  |\Tor{\Sigma_2}{\C_\xi}|^2
  = \pm
  \left(
  \lim_{t \to 1}
  \frac{\twistedAlex[\sqrt{-1}\, t]{\knotexterior}{\rho}}{t^2-1}
  \right)^2\cdot
  \left.
  \frac{I_\lambda^2 -4}{I_{\widehat{\mu}}^2 -4}
  \left(\frac{d I_{\widehat{\mu}}}{dI_\lambda}\right)^2
  \right|_{\chi = \chi_{\rho'}}.  
  $$
\end{proof}

\begin{remark}
  The assumption that $C_*(\Sigma_2;\C_\xi)$ is acyclic is equivalent to
  the condition for an irreducible metabelian representation $\rho'$ to be $\widehat{\mu}$-regular
  from Propositions~\ref{prop:twistedhomology_coveringspaces},
  \ref{prop:basis_twistedhomology_coverExterior} \&~\ref{prop:decomp_twistedhomology_coverExterior}.
  Actually, this is equivalent to that $\rho'$ is $\mu$-regular
  since the base change formula~\eqref{eqn:square_det_basechange_matrix}
  between $\widehat{\mu}$ and $\mu$ is $4$.
\end{remark}

\begin{remark}
  The assumption of Lemma~\ref{lemma:determinant_MV} requires
  that $\rho'$ is $\lambda$-regular.
  It is a sufficient condition
  for each factor in the right hand side of Eq.~\eqref{eqn:Main_Formula} to be well--defined as
  a complex number.
\end{remark}

\begin{remark}
  The left hand side of Eq.~\eqref{eqn:Main_Formula} is well--defined as a topological invariant of $\Sigma_2$
  when $H_*(\Sigma_2;\C_\xi)=\bm{0}$.
  Theorem~\ref{thm:main_theorem} shows that the product in the right hand side of~\eqref{eqn:Main_Formula}
  is independent of the choice of basis of $H_*(\knotexterior;\sll)$, 
  even though each factor depends on such a choice, 
  and gives a topological invariant of $\Sigma_2$, which coincides with the square of absolute value of
  Reidemeister torsion defined by a non--trivial $\GL_1(\C)$-representation $\xi$.
\end{remark}

\section{Application to two--bridge knots}
\label{sec:application}
We discuss the twisted Alexander polynomial for irreducible metabelian representations and 
the rational function 
$\frac{I_\lambda^2 -4}{I_{\widehat{\mu}}^2 -4}
\left(\frac{d I_{\widehat{\mu}}}{dI_\lambda}\right)^2$ 
on the character varieties
for two--bridge knots.
Subsection~\ref{subsec:Torus_knots} deals with the non--hyperbolic two--bridge knots
and we will discuss the case of hyperbolic two--bridge knots 
in Subsection~\ref{subsec:hyperbolic_twobridge}.

\subsection{Non--hyperbolic two bridge knots}
\label{subsec:Torus_knots}
From~\cite{HatcherThurston85} it is known that two--bride knots have no satellite knots
and hence every non--hyperbolic two--bridge knot is a torus knot of type $(2, q)$ where
$q$ is an odd integer.
We give explicit forms for the Reidemeister torsion $\Tor{\knotexterior}{\sll}$ 
and the rational function 
$\frac{I_\lambda^2 -4}{I_{\widehat{\mu}}^2 -4}
\left(\frac{d I_{\widehat{\mu}}}{dI_\lambda}\right)^2$ 
on the character varieties for $(2, q)$-torus knots.
For simplicity, we assume that $q$ is a positive odd integer.

We start with the description of character varieties of torus knots,
according to the lecture note of D.~Johnson~\cite{Johnson:unpublished}.
Here we adopt the following presentation of the knot group for the $(p,q)$-torus knot $K$:
$$
\knotgroup = \langle x, y \,|\, x^p=y^q \rangle.
$$

\begin{figure}[ht]
  \begin{center}
    \includegraphics[scale=.4]{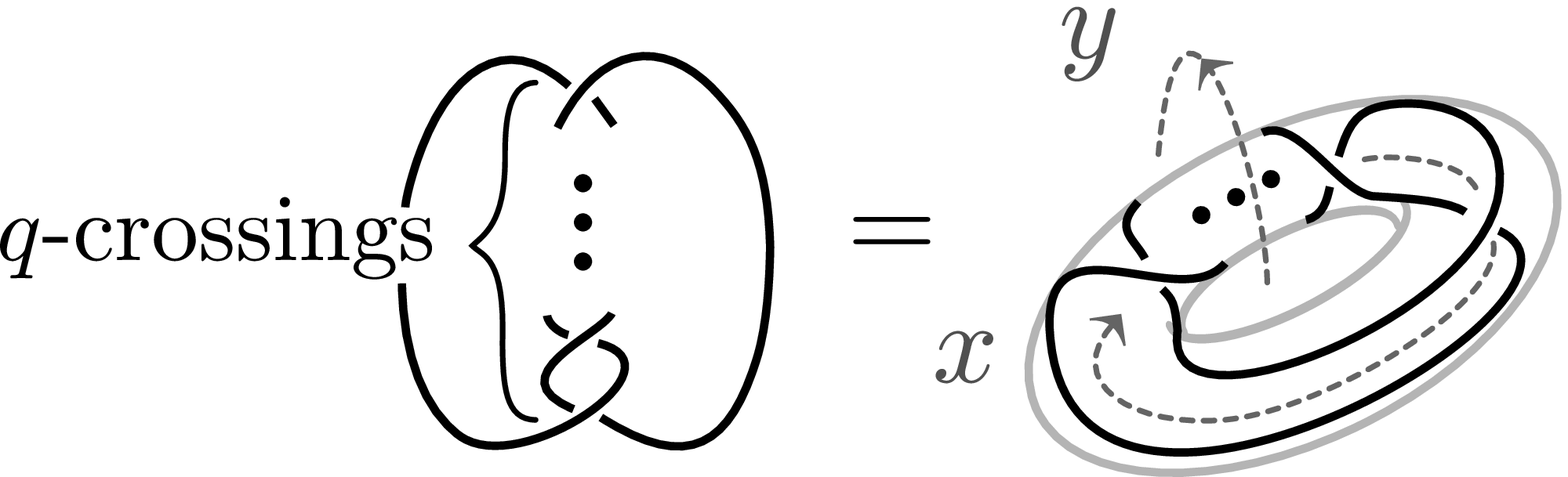}
  \end{center}
  \caption{The diagram of $(2, q)$-torus knot $K$ ($q>0$) and generators $x$ and $y$ of $\knotgroup$.}
  \label{fig:2_q_torusknot}
\end{figure}

\begin{proposition}[\cite{Johnson:unpublished},
  Proposition~$3.7$ in~\cite{KitanoMorifuji:TwistedAlexTorusKnots}]
  \label{prop:chractervar_torusknot}
  Let $K$ be the $(p, q)$-torus knot.
  Then $X^{\rm irr}(E_K)$ consists of $(p-1)(q-1)/2$ components,
  which are determined by the following data, denoted by $X_{a, b}$:
  \begin{enumerate}
  \item $0 < a < p$, $0 < b < q$.
  \item $a \equiv b$ mod $2$.
  \item For every $\chi \in X_{a, b}$, we have that 
    $\chi(x) = 2\cos\left(\frac{\pi a}{p}\right)$ and
    $\chi(y) = 2\cos\left(\frac{\pi b}{q}\right)$.
    Moreover irreducible representation $\rho$, whose character is $\chi$, sends $x^p (= y^q)$
    to $(-\I)^a$.
  \item  Each component $X_{a, b}$ is parametrized by $I_{\mu}$ 
    and the characters $\chi$ satisfying that $I_\mu (\chi) = 2 \cos (\pi (ra/p \pm sb/q))$
    are reducible where $\mu$ denotes the meridian given by $x^{-r} y^s$ and
    $r$ and $s$ are integers such that $ps - qr=1$.
  \end{enumerate}
  In particular, every $X_{a, b}$ has the complex dimension one.
\end{proposition}

\begin{corollary}
  If $K$ is the $(2, q)$-torus knot, then the irreducible components $X^{\rm irr}(\knotexterior)$
  consists of $(q-1)/2$ components $X_{1, b}$, which has the local coordinate $I_{\mu}$,
  for odd integer $b$ satisfying $0 < b < q$.
\end{corollary}

From now on, the symbol $K$ denotes $(2, q)$-torus knot in this subsection.
Every character of irreducible metabelian representation is contained in the subset defined by $I_{\mu} = 0$
from Proposition~\ref{prop:correspondence_abel_metabel}.
\begin{corollary}
  It follows from $|\Delta_K(-1)|=q$ that 
  each component $X_{1, b}$ has only one irreducible metabelian representation as the origin 
  under the local coordinate $I_{\mu}$.
\end{corollary}

\begin{proposition}[Theorem~$4.2$ in~\cite{KitanoMorifuji:TwistedAlexTorusKnots}]
  \label{prop:twistedAlex_TorusKnots}
  We suppose that the character of an irreducible metabelian representation $\rho$ of $\knotgroup$ is
  contained in the component $X_{1, b}$.
  Then the twisted Alexander polynomial is expressed as
  $$
  \twistedAlex{\knotexterior}{\rho}
  = (t^2 + 1)
  \prod_{\substack{
      1 \leq \ell \leq (q-1)/2 \\
      \ell \not = (q-b)/2
  }}
  (t^2 + \zeta_q^{\ell})(t^2 + \zeta_q^{-\ell})
  $$
  where $\zeta_q = e^{2 \pi \sqrt{-1}/q}$. 
\end{proposition}
We can compute directly
the rational function $\twistedAlex[\sqrt{-1} t]{\knotexterior}{\rho}/ (t^2 -1)$ in Theorem~\ref{thm:main_theorem}
and its evaluation at $t=1$.
\begin{lemma}
  We suppose that the character an irreducible metabelian representation $\rho$ 
  of $\knotgroup$ is contained in $X_{1, b}$.
  The square of $P(1) = \lim_{t \to 1} \twistedAlex[\sqrt{-1} t]{\knotexterior}{\rho}/ (t^2 -1)$
  in Theorem~\ref{thm:main_theorem} is expressed as
  $$
  P(1)^2
  =
  \left(
  \frac{q}{4 \sin^2 (j \pi/q)}
  \right)^2
  $$
  where $j = (q-b)/2$.
\end{lemma}
\begin{proof}
  By Proposition~\ref{prop:twistedAlex_TorusKnots} and Lemma~\ref{lemma:polynomial_P},
  we can express as $P(1)$ as
  $$
  P(1) =
  - \prod_{
    \substack{
      1 \leq \ell \leq (q-1)/2 \\
      \ell \not = j
  }}
  (1 - \zeta_q^\ell)(1 - \zeta_q^{-\ell}).
  $$
  From the relation that $\prod_{\ell=1}^{(q-1)/2 } (t-\zeta_q^\ell)(t-\zeta_q^{-\ell}) = 1+t+ \cdots + t^{q-1}$,
  we can rewrite the above equality as
  $$
  P(1)
  = \frac{-q}{(1-\zeta_q^j)(1-\zeta_q^{-j})}
  = \frac{-q}{4 \sin (j \pi/q)}.
  $$
\end{proof}

We proceed to compute the rational function
$\frac{I_\lambda^2 -4}{I_{\widehat{\mu}}^2 -4}
\left(\frac{d I_{\widehat{\mu}}}{dI_\lambda}\right)^2$.
We will see that this function is constant whole $X^{\rm irr}(\knotexterior)$.

\begin{lemma}
  For $(2, q)$-torus knot $K$, we have 
  $$
  \frac{I_\lambda^2 -4}{I_{\widehat{\mu}}^2 -4}
  \left(\frac{d I_{\widehat{\mu}}}{dI_\lambda}\right)^2
  = \frac{1}{q^2}
  $$
  on the whole set $X^{\rm irr}(\knotexterior)$.
\end{lemma}
\begin{proof}
  It is known that the preferred longitude $\lambda$ is represented by
  $\mu^{2q}x^{-2}$ in the torus knot group $\knotgroup$.
  We can deduce the relation between
  the regular functions $I_\lambda$ and $I_{\widehat{\mu}}$
  on the character variety.
  Let $\widehat{M}^{\pm 1}$ be the eigenvalues of matrix
  corresponding to $\widehat{\mu}$ under an irreducible representation.
  By Proposition~\ref{prop:chractervar_torusknot}, the element $x^2$ is
  sent to $-\I$ by every irreducible representation.
  Hence the function $I_\lambda$ is given by $-(\widehat{M}^q + \widehat{M}^{-q})$.
  Since $\widehat{M}^{\pm 1}$ is given by $\frac{1}{2}\Big(I_{\widehat{\mu}} \pm \sqrt{I_{\widehat{\mu}}^2 -4} \Big)$,
  we can express the function $I_{\lambda}$ as
  $$
  I_{\lambda} =
  -\frac{1}{2^q}\left\{
  \left(
  I_{\widehat{\mu}} + \sqrt{I_{\widehat{\mu}}^2 -4}
  \right)^q
  + 
  \left(
  I_{\widehat{\mu}} - \sqrt{I_{\widehat{\mu}}^2 -4}
  \right)^q
  \right\}.
  $$
  The derivative of $I_{\lambda}$ by $I_{\widehat{\mu}}$ is given by
  \begin{align*}
    \frac{d I_\lambda}{d I_{\widehat{\mu}}}
    &=
    \frac{-q}{\sqrt{I_{\widehat{\mu}}^2 -4}}
    \cdot
    \frac{1}{2^q}
    \left\{
    \left(I_{\widehat{\mu}} + \sqrt{I_{\widehat{\mu}}^2 -4}\right)^q
    - \left(I_{\widehat{\mu}} - \sqrt{I_{\widehat{\mu}}^2 -4}\right)^q
    \right\} \\
    &=
    \frac{-q}{\sqrt{I_{\widehat{\mu}}^2 -4}}
    (\widehat{M}^q - \widehat{M}^{-q}).
  \end{align*}
  Taking square of both sides, we have the equality:
  \begin{align*}
    \left(
    \frac{d I_\lambda}{d I_{\widehat{\mu}}}
    \right)^2
    &=
    \frac{q^2}{I_{\widehat{\mu}}^2 -4} \cdot ((\widehat{M}^q + \widehat{M}^{-q})^2 -4)\\
    &=
    \frac{q^2}{I_{\widehat{\mu}}^2 -4} \cdot (I_\lambda^2 -4),
  \end{align*}
  which gives the desired relation.
\end{proof}
We recall the Reidemeister torsion for $\Sigma_2$. It is known that
the double branched cover along $(2, q)$-torus knot is the lens space of type $(q, 1)$.
We denote by $\gamma$ a generator of $\piDB=\pi_1(L(q, 1))$, \ie
$$
\piDB = \langle \gamma \,|\, \gamma^q = 1 \rangle.
$$
\begin{proposition}
  \label{prop:Rtorsion_LensSpace}
  Let $K$ be a $(2, q)$-torus knot and
  $\xi$ a non--trivial homomorphism from $\piDB$ to $\GL_1(\C)$ sending $\gamma$ to $e^{2\pi\sqrt{-1}/q}$.
  We denote by $\xi^j$ the homomorphism of $\piDB$ sending $\gamma$ to $e^{2\pi\sqrt{-1}j/q}$.
  Then The Reidemeister torsion for $\Sigma_2$ and $\xi^j \oplus \xi^{-j}$ is expressed as
  $$
  \Tor{\Sigma_2}{\C_{\xi^j} \oplus \C_{\xi^{-j}}}
  = \frac{1}{\left( 4 \sin^2 (j \pi / q)\right)^2}.
  $$
\end{proposition}
\begin{proof}
  The Reidemeister torsion $\Tor{\Sigma_2}{\C_{\xi^j} \oplus \C_{\xi^{-j}}}$
  for the direct sum representation turns into the product:
  $$
  \Tor{\Sigma_2}{\C_{\xi^j}} \cdot \Tor{\Sigma_2}{\C_{\xi^{-j}}}.
  $$
  We can see that $\xi^{-j}$ is the same as $\overline{\xi^j}$ since
  the image by $\xi$ is contained in the unit circle of $\C$.
  Hence the Reidemeister torsion for $\xi^{-j}$ is the complex conjugate of
  $\Tor{\Sigma_2}{\C_{\xi^j}}$.
  It is known that the Reidemeister torsion $\Tor{\Sigma_2}{\C_{\xi^j}}$ is given by
  $$
  \frac{
    (\zeta_q^j)^\ell
  }{
    (1-\zeta_q^j)^2
  }
  $$
  where $\zeta_q = e^{2\pi\sqrt{-1}/q}$ and $\ell \in \Z$.
  Therefore $\Tor{\Sigma_2}{\C_{\xi^j} \oplus \C_{\xi^{-j}}}$ turns out to be
  $$
  \frac{
    1
  }{
    (1-\zeta^j_q)^2(1-\zeta^{-j}_q)^2
  }
  $$
  which completes the proof.
\end{proof}

\begin{remark}
  For precise presentations of $\piDB$ and $\xi$,
  we refer to Lemma~\ref{lemma:metabelian_twobridge_parametrization}.
  The conjugacy class $[\rho]$ of metabelian representations in $X_{1, b}$ corresponds to
  $[\rho_k]$ for $k = (q-b)/2$ in Lemma~\ref{lemma:metabelian_twobridge_parametrization}.
\end{remark}

We can check our formula in Theorem~\ref{thm:main_theorem}.
We choose an irreducible metabelian representation $\rho$
whose character is contained in the component $X_{1, b}$,
where $1 < b < q$ and $b$ is odd.
The right hand side of~\eqref{eqn:Main_Formula} turns into the product
$$
P(1)^2 \cdot
\frac{
  I_\lambda^2 -4
}{
  I_{\widehat{\mu}}^2 -4
}
\left(
\frac{
  d I_{\widehat{\mu}}
}{
  dI_\lambda}
\right)^2
=
\left(
\frac{
  q
}{
  4 \sin^2 (j \pi / q )
}
\right)^2
\frac{1}{q^2}
=
\left(
\frac{
  1
}{
  4 \sin^2 (j \pi / q )
}
\right)^2
$$
where $j = (q-b)/2$.
This value agrees with the Reidemeister torsion $\Tor{\Sigma_2}{\C_{\xi^j} \oplus \C_{\xi^{-j}}}$
in Proposition~\ref{prop:Rtorsion_LensSpace}.

\subsection{Hyperbolic two bridge knots}
\label{subsec:hyperbolic_twobridge}
We will show that every irreducible metabelian representation is 
$\lambda$-regular for hyperbolic tow--bridge knot groups.
This guarantees that the torsion $\Tor{\coverExterior}{\sll}$ is non--zero and
$\TorMV{\hbox{\ref{eqn:Mayer_Vietoris}}}$ does not diverge.
We will derive the $\lambda$-regularity from the Property L of hyperbolic two--bridge knots.

We recall the definition of Property L of knots, which follows the paper~\cite{BoileauBoyerReidWang:Simon}
by M.~Boileau, S.~Boyer, A.~Reid and S.~Wang.
\begin{definition}
  Let $K$ be  a knot in $S^3$. We denote by $K(0)$ the manifold obtained from by a longitudinal surgery
  on $K$. We say that $K$ has Property L if the $\SL$-character variety of the manifold $K(0)$ contains
  only finitely many characters of irreducible representations.
\end{definition}

\begin{proposition}[Proposition~$1.10$ in \cite{BoileauBoyerReidWang:Simon}]
  \label{prop:Property_L}
  If $K$ be a hyperbolic two--bridge knot, then $K$ has Property L.
\end{proposition}

Property L means that the character variety $X(K(0))$ consists of only points. 
This property gives us the $\lambda$-regularity of an irreducible metabelian representation.

\begin{proposition}
  \label{prop:ciriterion_regularity}
  Suppose that $\rho$ is an irreducible metabelian $\SL$-representation of $\knotgroup$.
  Then $\rho$ is $\lambda$-regular if and only if
  $\dim_{\C} X(K(0)) = 0$ at the character of $\rho$ and its character is a smooth point of $X(K(0))$.
\end{proposition}

Note that every metabelian $\SL$-representation sends the longitude $\lambda$ to $\I$.
This means that all metabelian representations induce homomorphism
from the fundamental group $\pi_1(K(0)) = \knotgroup / \nclos{\lambda}$ into $\SL$.
We use the symbol $\bar{\rho}$ for the induced representation of $\pi_1(K(0))$ by $\rho$.

\begin{corollary}
  If $K$ is a hyperbolic two--bridge knot,
  then all irreducible metabelian $\SL$-representations of $\knotgroup$ is $\lambda$-regular.
\end{corollary}
\begin{proof}
  It follows from Hilbert's Nullstellensatz and Propositions~\ref{prop:Property_L}
  that each component of the character variety $X(K(0))$ is defined by the maximal ideal $(X_1 - a_1, X_2 -a_2, \ldots, )$.
  By direct calculation,
  we can check that every conjugacy class of irreducible metabelian representations satisfies the conditions to be $\lambda$-regular
  in Proposition~\ref{prop:ciriterion_regularity}.
\end{proof}

\begin{proof}[Proof of Proposition~\ref{prop:ciriterion_regularity}]
  We assume that an $\SL$-representation $\rho$ 
  is irreducible and metabelian.
  First it follows from a Lin presentation (we refer to~\cite[Lemma~$2$]{NagasatoYamaguchi})
  that the induced homomorphism of $\knotgroup / \nclos{\lambda}$ is also irreducible.
  From the construction of $K(0) = \knotexterior \cup D^2 \times S^1$, in which the meridian circle
  $\partial D^2 \times \{*\}$ is glued along the longitude $\lambda$, we have 
  the Mayer--Vietoris exact sequence of homology group with the coefficient $\sll$:
  $$
  \cdots \to H_{i+1}(K(0)) \to
  H_i(\boundaryTorus) \to
  H_i(\knotexterior) \oplus H_i(D^2 \times S^1) \to
  H_i(K(0)) \to \cdots
  $$
  Since both representations of $\knotgroup$ and $\pi_1(K(0))$ are irreducible,
  in particular which are non--abelian, 
  it follows 
  that $H_0(\knotexterior;\sll)=\bm{0}$ and $H_0(K(0);\sll) \simeq H_3(K(0);\sll) =\bm{0}$.
  Since $Ad \circ \rho (\mu)$ has eigenvalues $\pm 1$ and the multiplicity of $-1$ is $2$,
  we choose $P^\rho \in \sll$ as an eigenvector of the eigenvalue $1$.
  Then the twisted homology groups $H_*(\boundaryTorus;\sll)$ and $H_*(D^2 \times S^1;\sll)$ is expressed as
  \begin{align*}
    H_*(\boundaryTorus;\sll) &\simeq
    \begin{cases}
      {}_{\C}\langle P^\rho \otimes T^2 \rangle & (*=2) \\
      {}_{\C}\langle P^{\rho} \otimes \mu, P^{\rho} \otimes \lambda \rangle & (*=1) \\
      {}_{\C}\langle P^\rho \otimes \bpt \rangle & (*=0)
    \end{cases}, \\
    H_*(D^2 \times S^1;\sll) &\simeq
    \begin{cases}
      {}_{\C}\langle P^{\rho} \otimes \mu \rangle & (*=1) \\
      {}_{\C}\langle P^\rho \otimes \bpt \rangle & (*=0)
    \end{cases}
  \end{align*}
  Now we start with if part.
  The assumption on $X(K(0))$ deduces that
  $$
  0= \dim_{\C} X(K(0)) = \dim_{\C} T^{Zar}_{\chi_{\bar \rho}} X(K(0)) = \dim_{\C} H_1(K(0);\sll)
  $$
  where $T^{Zar}_{\chi_{\bar \rho}} X(K(0))$ denotes the Zariski tangent space at the character $\chi_{\bar \rho}$.
  The second equality follows from that the character $\chi_{\bar \rho}$ is a smooth point and
  the third equality from that $\bar \rho$ is irreducible.
  We have also $H_2(K(0);\sll)=0$ by Poincar\'e duality and the universal coefficient theorem.
  Substituting these results in the Mayer--Vietoris exact sequence, we have
  \begin{equation}
    \label{eqn:basis_lambda_regular}
    H_2(\knotexterior;\sll) =
    {}_{\C}\langle P^\rho \otimes T^2 \rangle,\quad 
    H_1(\knotexterior;\sll) =
    {}_{\C}\langle P^\rho \otimes \lambda \rangle,
  \end{equation}
  which means $\rho$ is $\lambda$-regular.

  If $\rho$ is $\lambda$-regular, then we can express the twisted homology group $H_*(\knotexterior;\sll)$
  as Eq.~\eqref{eqn:basis_lambda_regular}.
  Using Mayer--Vietoris exact sequence, we obtain that $H_*(K(0);\sll)=0$.
  The following inequality of dimensions implies that $\dim_\C X(K(0)) =0$ and
  the character $\chi_{\bar \rho}$ is a smooth point:
  $$
  0 \leq \dim_{\C} X(K(0)) \leq \dim_{\C} T^{Zar}_{\chi_{\bar \rho}} X(K(0)) = \dim_{\C} H_1(K(0);\sll).
  $$
\end{proof}

For every two--bridge knot,
each factor in the right hand side of Eq.~\eqref{eqn:Main_Formula} is defined as non--zero number
since $H_*(\Sigma_2; \C_\xi) = 0$.
It is known that isotopy types of two--bridge knots correspond to
topological types of double branched covers, \ie
lens spaces.
The set of Reidemeister torsions for lens spaces and $\xi^k \oplus \xi^{-k}$ distinguishes 
topological types of lens spaces.
To be more precisely, Reidemeister torsion distinguishes the PL--isomorphism types of 
lens space, however since we have no differences between the PL--category and the topological category
in the three--dimensional topology, we can say that the Reidemeister torsion is a topological invariant.
Together with Theorem~\ref{thm:main_theorem},
we can conclude the following theorem
regarding the set of rights hand side in Eq.~\eqref{eqn:Main_Formula} as an invariant of two--bridge knots.
\begin{theorem}
  \label{thm:distinguish_twobridgeknots}
  Let $K$ be a two--bridge knot.
  For every irreducible metabelian $\SL$-representation of $\knotgroup$,
  we have
  \begin{itemize}
  \item $\twistedAlex[\sqrt{-1} t]{\knotexterior}{\rho} / (t^2-1)$ turns into a Laurent polynomial
    and gives a non--zero complex number at $t=1$;
  \item the rational function
    $\frac{I_\lambda^2 -4}{I_{\widehat{\mu}}^2 -4}
    \left(\frac{d I_{\widehat{\mu}}}{dI_\lambda}\right)^2$ gives a non--zero complex number
    at the character of $\rho$.
  \end{itemize}
  
  Moreover
  the set of the following products as in Eq.~\eqref{eqn:Main_Formula}:
  $$
  \left|
  \left(
    \lim_{t \to 1} \frac{\twistedAlex[\sqrt{-1}\,t]{\knotexterior}{\rho}}{t^2 -1}
    \right)^2
    \left.
    \frac{I_\lambda^2 -4}{I_{\widehat{\mu}}^2 -4}
    \left(\frac{d I_{\widehat{\mu}}}{dI_\lambda}\right)^2
    \right\rvert_{\chi = \chi_{\rho'}}
  \right|
  $$
  for $(|\Delta_K(-1)| -1)/2$ characters of irreducible metabelian representations
  distinguishes isotopy classes of two--bridge knots up to mirror images.
\end{theorem}
\begin{proof}
  Without loss of generality, we can assume that
  two--bridge knots $K$ and $K'$ have the same knot determinant
  $|\Delta_K(-1)|=|\Delta_{K'}(-1)|=p$.
  This means that the corresponding double branched covers are $L(p, q)$ and $L(p, q')$.
  The Reidemeister torsion of $L(p, q)$ with $\xi \oplus \xi^{-1}$ for some $\GL_1(\C)$-representation
  of $\pi_1(L(p, q))$ is expressed as
  $$
   \big((\zeta -1)(\zeta^r-1)(\zeta^{-1} -1)(\zeta^{-r}-1)\big)^{-1}
  $$
  where $\zeta$ is a $p$-th root of unity and $r$ satisfies that $qr \equiv 1$ mod $p$.
  Since we suppose that the set of the Reidemeister torsions for $L(p, q)$ coincides with
  that for $L(p, q')$. Hence we can find the Reidemeister torsion for $L(p, q')$ satisfies that
  $$
  (\zeta -1)(\zeta^r-1)(\zeta^{-1} -1)(\zeta^{-r}-1)
  =(\zeta^{k} -1)(\zeta^{kr'}-1)(\zeta^{-k} -1)(\zeta^{-kr'}-1)
  $$
  Hence we have the following equivalent relations
  from the same argument in~\cite[Proof of Theorem~$10.1$]{Turaev:2000} and Franz Independent Lemma:
  $$
  1 \equiv \pm k,\, r \equiv \pm kr' \quad \hbox{or} \quad
  1 \equiv \pm kr',\, r \equiv \pm k \quad (\hbox{mod}\, p).
  $$
  This condition derives that $q' \equiv \pm q$ or $q' \equiv \pm q^{\pm 1}$.
  When $q'$ is congruent to $q^{\pm 1}$, two lens spaces are orientation preserving homeomorphic
  which means that $K$ and $K'$ are isotopic.  
  In the other case, two lens spaces are orientation reversing homeomorphic and
  $K$ is isomorphic to the mirror image of $K'$.
\end{proof}

Last, we discuss how to compute the twisted Alexander polynomial 
and the rational function on the character varieties for hyperbolic two--bridge knot exteriors.
To compute them,
we need explicit forms for
irreducible metabelian representations from a hyperbolic two--bridge knot group into $\SL$.
It is useful for this purpose
to apply Riley's method to make non--abelian $\SL$-representations of two--bridge knot groups.
Let $K$ be a hyperbolic two--bridge knot. 
We start with the following presentation of $\knotgroup$
obtained from a Wirtinger presentation: 
\begin{equation}
  \label{eqn:presentation_twobridgeknotgroup}
  \knotgroup =
  \langle x, y \,|\, wx = yw \rangle 
\end{equation}
where $x$ and $y$ are meridians and $w$ is a word in $x$ and $y$ as in Figure~\ref{fig:twobridge}.

\begin{figure}[ht]
  \begin{center}
    \includegraphics[scale=.5]{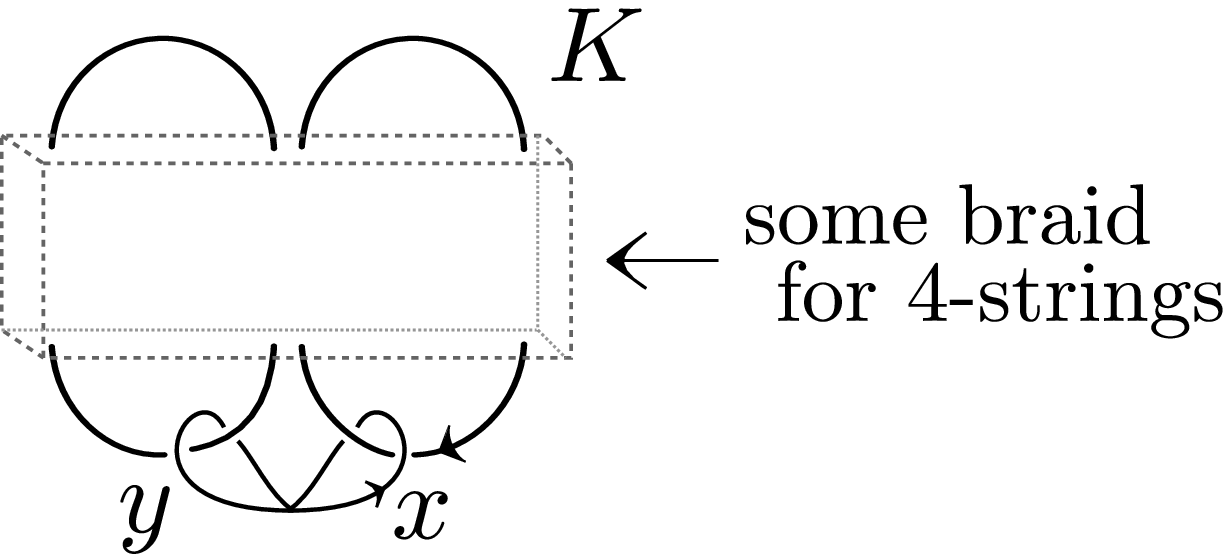}
  \end{center}
  \caption{The generators $x$ and $y$ for a two--bridge knot $K$}
  \label{fig:twobridge}
\end{figure}

By the method~\cite{Riley}, we have the following explicit form
of irreducible metabelian $\SL$-representation of $\knotgroup$.
\begin{lemma}[Theorem~$3$ in~\cite{NagasatoYamaguchi}]
  \label{lemma:metabelian_twobridge_parametrization}
  Let $p$ denote $|\Delta_K(-1)|$.
  For all conjugacy classes of irreducible metabelian $\SL$-representations,
  we can choose the following representatives:
  $$
  \{[\rho_k] \,|\, k = 1, \ldots, (p-1)/2\}
  $$
  where $\rho_k$ is an irreducible $\SL$-representation given by the correspondence:
  $$
  x \mapsto
  \begin{pmatrix}
    \sqrt{-1} & -\sqrt{-1} \\
    0 & -\sqrt{-1}
  \end{pmatrix},
  \quad
  y \mapsto
  \begin{pmatrix}
    \sqrt{-1} & 0 \\
    -\sqrt{-1}u_k & -\sqrt{-1}
  \end{pmatrix} 
  \, (u_k = (e^{k \pi \sqrt{-1}/p} - e^{-k \pi \sqrt{-1}/p})^2).
  $$
  Moreover the representative $\rho_k$ corresponds to the $\GL_1(\C)$-representation $\xi^k$:
  $$
  \xi^k \co \piDB = \langle \gamma \,|\, \gamma^p = 1\rangle \ni \gamma
  \mapsto e^{2k\pi\sqrt{-1}/p} \in \GL_1(\C)
  $$
  where $\gamma$ is a lift of $xy^{-1}$ in $\knotgroup$ into $\piDB$. 
\end{lemma}

\subsection*{The twisted Alexander polynomial}
We simply denote by $\rho$ an irreducible metabelian representation
in Lemma~\ref{lemma:metabelian_twobridge_parametrization}.
We follow the definition and computation of the twisted Alexander polynomial along Wada's way.
The twisted Alexander polynomial for $\knotexterior$ and $\rho$ is expressed as
$$
\twistedAlex{\knotexterior}{\rho}
=
\frac{
  \det \left(
  \alpha \otimes \rho
  \left(
  \frac{d}{dx} wxw^{-1}y^{-1}
  \right)
  \right)
}{
  \det (t \,\rho(y) -\I)
}
$$
where $\frac{d}{dx}$ denotes the Fox differential by $x$.

The Fox differential $\frac{d}{dx} wxw^{-1}y^{-1}$ turns into $w + (1-wxw^{-1})\frac{d}{dx}w$.
The numerator is a Laurent polynomial whose coefficients are polynomials in $\C[u_k]$.
The denominator $\det(t \rho(t) - \I)$ turns into $t^2+1$.

\subsection*{The rational function}
To express the rational function 
$\frac{I_\lambda^2 -4}{I_{\widehat{\mu}}^2 -4}
\left(\frac{d I_{\widehat{\mu}}}{dI_\lambda}\right)^2$,
we need to find which representative gives another metabelian $\rho'$ whose adjoint representation
is conjugate to $\signrep \oplus \sqrtrep \rho$.
\begin{proposition}
  We assume that $\rho$ denotes $\rho_k$ in Lemma~\ref{lemma:metabelian_twobridge_parametrization}.
  Then we can choose another metabelian representation $\rho'$
  in Theorem~\ref{thm:distinguish_twobridgeknots}
  as $\rho_{k'}$ in Lemma~\ref{lemma:metabelian_twobridge_parametrization}
  where $k'$ is an integer in $\{1, \ldots, (p-1)/2\}$ satisfying
  $2k' \equiv k$ or $2k' \equiv -k$ {\rm mod} $p$.
\end{proposition}
\begin{proof}
  We compare the traces of the images of $xy^{-1}$ by $Ad \circ \rho'$ and
  $\signrep \oplus \sqrtrep \rho$.
  Under the notation in Lemma~\ref{lemma:metabelian_twobridge_parametrization},
  the traces of $\rho (xy^{-1})$ and $\rho'(xy^{-1})$
  are equal to $u_k + 2$ and $u_{k'}+2$.
  If an $\SL$-element $A$ has the eigenvalues $\zeta^{\pm 1}$, then
  the adjoint action has the eigenvalues $\zeta^{\pm 2}$ and $1$.
  Hence the trace of $Ad \circ \rho' (xy^{-1})$ is expressed as
  $$
  \trace \big( Ad \circ \rho' (xy^{-1}) \big)= 1+(u_{k'}+2)^2 -2 = (u_{k'}+2)^2 -1.
  $$
  Since the homology class of $xy^{-1}$ is trivial, the trace of the corresponding matrix by
  $\signrep \oplus \sqrtrep \rho$ is expressed as
  $$
  \trace \big( (-1)^{[xy^{-1}]} \oplus (\sqrt{-1})^{[xy^{-1}]} \rho(xy^{-1}) \big)
  = 1 + u_k +2.
  $$
  Since these traces are same, we have the equality:
  $$
  (e^{2k' \pi \sqrt{-1}/p} + e^{-2k' \pi \sqrt{-1}/p})^2
  = (e^{k \pi \sqrt{-1}/p} + e^{-k \pi \sqrt{-1}/p})^2
  $$
  We can rewrite this as
  $$
  \cos \frac{2k' \pi}{p}=
  \pm \cos \frac{k \pi}{p}
  $$
  which means that $2k' = k$ or $2k'=p-k$.
\end{proof}

It is known that $\lambda$ is expressed as $\overleftarrow{w} w x^{-2\sigma}$ 
in the presentation~\eqref{eqn:presentation_twobridgeknotgroup}
where $\overleftarrow{w}$ is the word of reverse order of $w$ and $\sigma$ is an integer
given by the sum of exponents of $x$ and $y$ in $w$.
Riley's construction of non--abelian representations follows from the correspondence:
\begin{equation}
  \label{eqn:Riley_correspondence}
  x \mapsto
  \begin{pmatrix}
    \sqrt{s} & 1/\sqrt{s} \\
    0 & 1/\sqrt{s}
  \end{pmatrix},
  \quad
  y \mapsto
  \begin{pmatrix}
    \sqrt{s} & 0 \\
    -u \sqrt{s} & 1/\sqrt{s}
  \end{pmatrix}
\end{equation}
and the equality:
\begin{equation}
  \label{eqn:Riley_poly}
  W_{1, 1} + (1-s)W_{1, 2} = 0
\end{equation}
where $W_{i, j}$ is the $(i, j)$-entry of the matrix corresponding to $w$.
The left hand side of Eq.~\eqref{eqn:Riley_poly} is a polynomial in $s + 1/s$ and $u$.
Hence we can regard the trace function $I_\lambda$ as a function on $s+1/s$, 
in particular, a function on $I_{\widehat{\mu}} = I_{\mu^2}$.

Let $\rho_{\sqrt{s}, u}$ denote the $\SL$-representation defined by the correspondence in 
Eq.~\eqref{eqn:Riley_correspondence}.
From Remark~\ref{remark:rest_adjoint_rep},
the trace function $I_\lambda$ gives $2$
at $(s, u) = (-1, u_{k'})$. More precisely, the behavior of $I_\lambda$ is expressed as follows.
\begin{proposition}
  We can express the trace function $I_\lambda$ as a function 
  $$
  I_\lambda - 2= - I_\mu^2 \cdot H(I_\mu) \quad \text{and} \quad H(0) \not = 0
  $$
  on $I_\mu$ near $(s, u) = (-1, u_{k'})$
  and the value of the rational function
  $\frac{I_\lambda^2 -4}{I_{\widehat{\mu}}^2 -4}
  \left(\frac{d I_{\widehat{\mu}}}{dI_\lambda}\right)^2$ at the character of $\rho'$
  as
  $$
  \left.
  \frac{I_\lambda^2 -4}{I_{\widehat{\mu}}^2 -4}
  \left(\frac{d I_{\widehat{\mu}}}{dI_\lambda}\right)^2
  \right|_{\chi = \chi_{\rho'}}
  =
  \frac{1}{H(0)}.
  $$
\end{proposition}
\begin{proof}
  The character of $\rho'$ corresponds to the pair $(-1, u_{k'})$.
  The function $I_\lambda -2$ has zero at $((-1, u_{k'}))$.
  Hence $I_\lambda -2$ is expressed as $I_\mu^k \cdot H(I_\mu)$ locally
  where $k$ is a positive integer and  $H(0) \not = 0$.
  Theorem~\ref{thm:distinguish_twobridgeknots} guarantees that the rational function
  $\frac{I_\lambda^2 -4}{I_{\widehat{\mu}}^2 -4}
  \left(\frac{d I_{\widehat{\mu}}}{dI_\lambda}\right)^2$
  is well--defined as a non--zero complex number, which deduces $k=2$.

  We can rewrite the rational function
  $\frac{I_\lambda^2 -4}{I_{\widehat{\mu}}^2 -4}
  \left(\frac{d I_{\widehat{\mu}}}{dI_\lambda}\right)^2$
  as
  \begin{align*}
  \frac{I_\lambda^2 -4}{I_{\widehat{\mu}}^2 -4}
  \left(\frac{d I_{\widehat{\mu}}}{dI_\lambda}\right)^2
  &=
  \frac{I_\lambda^2 -4}{I_\mu^2 -4} \frac{I_\mu^2 -4}{I_{\widehat{\mu}}^2 -4}
  \left(\frac{d I_{\widehat{\mu}}}{dI_\mu}\right)^2
  \left(\frac{d I_\mu}{dI_\lambda}\right)^2 \\
  &=
  4 
  \frac{I_\lambda^2 -4}{I_\mu^2 -4} 
  \left(\frac{d I_\mu}{dI_\lambda}\right)^2
  \end{align*}
  since $I_{\widehat{\mu}}$ is $I_\mu^2 -2$.
  Substituting $I_\lambda = -I_\mu^2 \cdot H(I_\mu)$, we can deduce the proposition.
\end{proof}
\begin{remark}
  If we choose a parameter as $I_{\widehat{\mu}} (=s+1/s)$, then we have 
  $I_\lambda = -(I_{\widehat{\mu}}+2) \widehat{H}(I_{\widehat{\mu}})$ where
  $\widehat{H}(-2) \not = 0$.
  The value of $\frac{I_\lambda^2 -4}{I_{\widehat{\mu}}^2 -4}
  \left(\frac{d I_{\widehat{\mu}}}{dI_\lambda}\right)^2$
  is given by $1/ \widehat{H}(-2)$.
\end{remark}

In the case that $K$ is the figure eight knot,
the Alexander polynomial $\Delta_K(t)$ is $t^2 -3t +1$. 
We have the two conjugacy classes of irreducible metabelian representations of $\knotgroup$
since $|\Delta_K(-1)|=5$.
For every irreducible metabelian representation $\rho$,
the twisted Alexander polynomial $\twistedAlex{\knotexterior}{\rho}$ is given by $t^2+1$
(we refer to~\cite[Section~$5.2$]{yamaguchi:twistedAlexMeta}).
From Example~$1$ in~\cite[Section~$4.4$]{Porti:1997}, we can see that
$I_\lambda - 2 = -I_\mu^2(-I_\mu^2+5)$.
Therefore every metabelian representation gives
the product in Theorem~\ref{thm:distinguish_twobridgeknots} as
$$
\left(\lim_{t \to 1} \frac{-t^2+1}{t^2-1}\right)^2
\frac{1}{-0^2+5}
= \frac{1}{5}.
$$
This coincides with the Reidemeister torsion of $\Sigma_2 = L(5,3)$ for $\xi^k \oplus \xi^{-k}$
$$
\left( \frac{1}{4\sin(2 k\pi/5) \sin( k \pi/5)} \right)^2
= \frac{1}{5}
$$
for $k=1,2$.

\section*{Acknowledgment}
The author wishes to express his thanks to Kunio Murasugi for drawing the author's attention to 
the classification of two--bridge knots by the twisted Alexander polynomial.
This research was supported by 
Research Fellowships of the Japan Society for the Promotion of Science for Young Scientists.

\bibliographystyle{amsalpha}
\bibliography{torsionDoubleBranchedCovers}
\end{document}